\documentclass[reqno, 12pt]{amsart}

\usepackage{amsmath}
\usepackage{amsfonts}
\usepackage{amssymb}
\usepackage{amsthm}
\usepackage{mathrsfs}
\usepackage{bbm}
\usepackage{graphicx, color}
\usepackage{tikz}
\usetikzlibrary{calc,decorations.markings}
\usetikzlibrary{arrows.meta}
\usetikzlibrary{positioning}
\usetikzlibrary{cd, intersections, calc, decorations.pathmorphing, arrows, decorations.pathreplacing}
\usepackage{color}
\usepackage{comment}

\usepackage{color}

\newcommand{\blue}[1]{\textcolor{blue}{#1}}

\usepackage{adjustbox}

\usepackage{bmpsize}

\usepackage{mathtools}

\usepackage{enumerate}

\definecolor{refkey}{rgb}{0.9451,0.2706,0.4941}
\definecolor{labelkey}{rgb}{0.9451,0.2706,0.4941}

\definecolor{mygreen}{rgb}{0,0.7,0.3}
\definecolor{myblue}{rgb}{0,0.50,1.20}
\definecolor{myorange}{rgb}{1,0.5,0.1}

\usepackage{subfiles}
\usepackage[colorlinks, linkcolor=blue, citecolor=red, urlcolor=cyan,pagebackref]{hyperref}

\allowdisplaybreaks

\numberwithin{equation}{section}

\usepackage{cleveref}
\crefname{thm}{Theorem}{Theorems}
\crefname{cor}{Corollary}{Corollaries}
\crefname{lem}{Lemma}{Lemmas}
\crefname{sublem}{Sublemma}{Sublemmas}
\crefname{prop}{Proposition}{Propositions}
\crefname{def}{Definition}{Definitions}
\crefname{example}{Example}{Examples}
\crefname{claim}{Claim}{Claims}
\crefname{conj}{Conjecture}{Conjectures}
\crefname{conv}{Notation}{Notations}
\crefname{rem}{Remark}{Remarks}
\crefname{rmk}{Remark}{Remarks}
\crefname{prob}{Problem}{Problems}
\crefname{figure}{Figure}{Figures}
\crefname{table}{Table}{Tables}
\crefname{section}{Section}{Sections}
\crefname{subsection}{Section}{Sections}
\crefname{appendix}{Appendix}{Appendices}
\crefname{introthm}{Theorem}{Theorems}
\crefname{introcor}{Corollary}{Corollaries}
\crefname{introconj}{Conjecture}{Conjectures}

\newtheorem{thm}{Theorem}[section]
\newtheorem{prop}[thm]{Proposition}
\newtheorem{cor}[thm]{Corollary}
\newtheorem{lem}[thm]{Lemma}

\theoremstyle{definition}

\newtheorem{conj}[thm]{Conjecture}

\theoremstyle{remark}
\newtheorem{rmk}[thm]{Remark}

\tikzset{
    squigarrow/.style={-{Classical TikZ Rightarrow[length=4pt]}, decorate, decoration={snake, amplitude=1.8pt, pre length=2pt, post length=3pt}}
}

\newcommand{\bC}{\mathbb C}

\newcommand{\bN}{\mathbb N}

\newcommand{\bQ}{\mathbb Q}

\newcommand{\bZ}{\mathbb Z}

\newcommand{\cC}{\mathcal C}

\newcommand{\cR}{\mathcal R}

\newcommand{\sS}{\mathscr{S}}
\newcommand{\sSq}{\mathscr{S}_q}

\newcommand{\DT}{{\rm DT}}

\newcommand{\Ha}{H_g^{\cC}}
\newcommand{\pHa}{\sSq(\partial H_g)_{\cC}}

\DeclareMathOperator{\Int}{\mathrm{Int}}

\DeclareMathOperator{\rankp}{\text{rank}_{\,\partial}}

\title[Finiteness conjecture for 3-manifolds obtained by attaching 2-handles]{Finiteness conjecture for 3-manifolds obtained from handlebodies by attaching 2-handles}

\author[Hiroaki Karuo]{Hiroaki Karuo}
\address{Hiroaki Karuo, Department of Mathematics, Gakushuin University, Mejiro, Toshima-ku, Tokyo, Japan.}
\email{hiroaki.karuo@gakushuin.ac.jp}

\author[Zhihao Wang]{Zhihao Wang}
\address{Zhihao Wang, School of Physical and Mathematical Sciences, Nanyang Technological University, 21 Nanyang Link Singapore 637371}
\email{ZHIHAO003@e.ntu.edu.sg}
\address{Bernoulli Institute, University of Groningen, P.O. Box 472, 9700 AK Groningen, The Netherlands}
\email{wang.zhihao@rug.nl}

\date{}

\begin{document}
\maketitle

\begin{abstract}
 We study a generalized Witten's finiteness conjecture for the skein modules of oriented compact $3$-manifolds with boundary. 
We formulate an equivalent version of the generalized finiteness conjecture using handlebodies and 2-handles, and prove the conjecture for some classes with the handlebodies of genus $2$ and $3$ using the equivalent version.
\end{abstract}

\setcounter{tocdepth}{1}
\tableofcontents

\section{Introduction}
Let $\cR$ be a commutative unital ring with a distinguished invertible element $q$ and $M$ be an oriented 3-manifold with (possibly empty) boundary $\partial M$. 
The \textbf{skein module} of $M$, denoted by $\sSq(M)$, is introduced as the $\cR$-module generated by isotopy classes of framed links (including the empty set) in $M$ subject to the following relations:
\begin{align*}
&({\rm A}) \begin{array}{c}\includegraphics[scale=0.18]{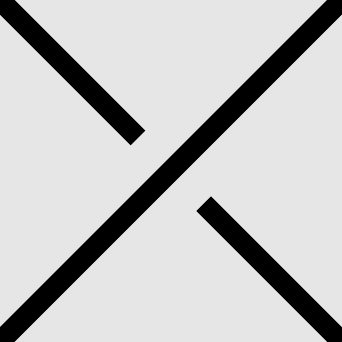}\end{array}=q\begin{array}{c}\includegraphics[scale=0.18]{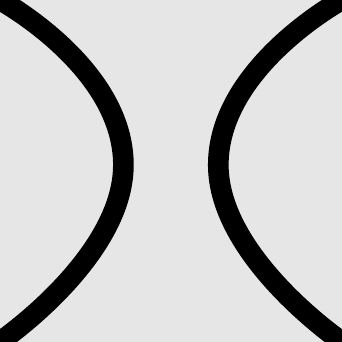}\end{array}+q^{-1}\begin{array}{c}\includegraphics[scale=0.18]{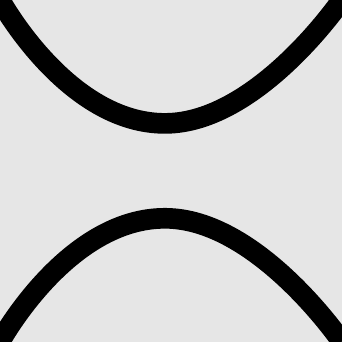}\end{array},\qquad ({\rm B}) \begin{array}{c}\includegraphics[scale=0.18]{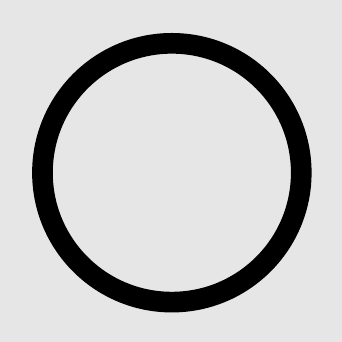}\end{array}=(-q^2-q^{-2})\begin{array}{c}\includegraphics[scale=0.18]{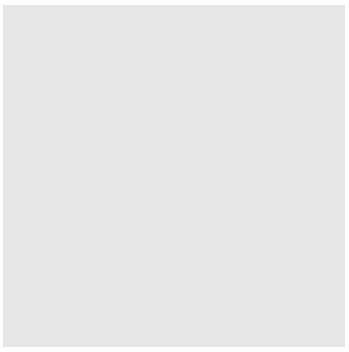}\end{array},
\end{align*}
where, in each relation, the local pictures are the intersection of framed links and an open 3-ball in $M$ and the framed links are the same except where shown.

For an oriented surface $\Sigma$, when we replace $M$ with $\Sigma\times [0,1]$, it allows multiplication on the skein module by stacking with respect to $[0,1]$. 
With this multiplication, the skein module of $\Sigma\times [0,1]$ is called the \textbf{skein algebra} of $\Sigma$, denoted by $\sSq(\Sigma)$. 

When we substitute $\pm1$ to $q$, we have 
$$
\begin{array}{c}\includegraphics[scale=0.18]{draws/crossing_posi.pdf}\end{array}
=\pm\begin{array}{c}\includegraphics[scale=0.18]{draws/resolution_posi.pdf}\end{array}
\pm\begin{array}{c}\includegraphics[scale=0.18]{draws/resolution_nega.pdf}\end{array}
=\begin{array}{c}\includegraphics[scale=0.18]{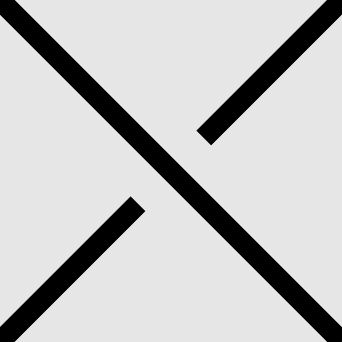}\end{array}. 
$$
This implies skein algebras are commutative at $q=\pm1$. 
In addition, skein modules are commutative algebras by taking the product of two framed links as a disjoint union of them. 
In particular, in the case when $\cR=\bC$, $\sS_{\pm1}(M)/\sqrt{0}$ is isomorphic to the coordinate ring of the ${\rm SL}_2\bC$-character variety of $\pi_1(M)$, where $\sqrt{0}$ is the nilradical \cite{Bar99}, \cite{Bul97}, \cite{PS00}. 

For any 3-manifold $M$,
the skein algebra $\sSq(\partial M)$ has an obvious action on $\sSq(M)$. 
We will use "$\cdot$" to denote this action.
More precisely, we identify a small closed regular neighborhood $U$ of $\partial M$ with $\partial M\times [0,1]$ such that $\partial M\times \{1\}$ is identified with $\partial M$. Then for any skeins $\alpha$ in $\partial M\times [0,1]$ and $\beta$ in $M$, we first push $\beta$ away from $U$, then define $\alpha\cdot\beta =\alpha\cup \beta$. 
With this action, $\sSq(M)$ can be regarded as a $\sSq(\partial M)$-module.

Skein modules have various features according to the ground ring $\cR$. 
For instance, when $\cR=\bQ(q)$ the field of rational functions in $q$ over $\bQ$, {\blue it was conjectured that skein modules of closed $3$-manifolds are finite dimensional; this was known as Witten's finiteness conjecture.}
After some partial results \cite{HP95}, \cite{HP93}, \cite{GH07}, \cite{Mro11}, \cite{Car17}, \cite{Det21} etc., the conjecture is solved affirmatively in \cite{GJS23}. 
\begin{thm}[Witten's finiteness conjecture {\cite{GJS23}}]\label{GJS}
For any oriented closed 3-manifold $M$, the skein module of $M$ has finite rank over $\bQ(q)$. 
\end{thm}

As a generalization of this conjecture, the following conjecture is formulated  by Detcherry \cite{Det21}.
\begin{conj}[Finiteness conjecture for 3-manifolds with boundaries]\label{conj}
For any oriented compact 3-manifold $M$ with (possibly empty) boundary $\partial M$, the skein module $\sSq(M)$ has  a finite subset which spans $\sSq(M)$ over the $\bQ(q)$-algebra $\sSq(\partial M)$, 
where if $\partial M=\emptyset$ then we regard $\sSq(\partial M)=\bQ(q)$. 
\end{conj}
\noindent Note that there is a stronger version of Conjecture \ref{conj} proposed by Detcherry \cite{Det21}. 

Except for closed 3-manifolds, it is known that Conjecture \ref{conj} holds for 
\begin{itemize}
    \item $F\times [0,1]$ with a compact oriented surface $F$ (by definition),
    \item Seifert fibered manifolds with orientable base and non-empty boundary \cite{AF22},
    \item the complement of a torus knot, a 2-bridge link or a ($-2,3,2n+1$)-pretzel knots in $S^3$ \cite{Le06}, \cite{Mar10}, \cite{LT14}, \cite{LT15}, 
\end{itemize}
and their proofs are based on diagrammatic techniques.

For a general 3-manifold $M$ with non-empty boundary, Conjecture \ref{conj} is more complicated than Witten's finiteness conjecture since we have to treat infinitely many skeins  near the boundary. 
While there are some techniques to understand which skeins are near the boundary, the techniques do not apply to general cases,  especially 3-manifolds having higher genus surfaces as their boundaries.
This paper aims that we attribute Conjecture \ref{conj} to an equivalent and tractable version with a handlebody and 2-handles (Theorem \ref{equal}). 
The idea comes from the fact that any oriented connected compact 3-manifolds are obtained from handlebodies by gluing 2-handles  up to 3-balls, see Section \ref{sec;equiv_conj}. 
With the equivalent version and diagrammatic techniques, we show Conjecture \ref{conj} affirmatively for some 3-manifolds with boundary and with Heegaard genus $2$ or $3$ (Theorem \ref{thm;example}). 
In particular, the boundary of each 3-manifold in the case of genus $3$ is the closed surface of genus $2$. This is the first non-trivial example for Conjecture \ref{conj} with higher genus boundary surfaces.
This implies that the examples are not covered by the listed previous results.

\paragraph{\textbf{Acknowledgements.}} 
The authors are grateful to Thang L\^e and Roland van der Veen for valuable comments. 
H. K. was partially supported by JSPS KAKENHI Grant Numbers JP22K20342, JP23K12976. 
Z. W. was supported by the NTU research scholarship and the PhD scholarship from the University of Groningen.
\section{Finiteness conjecture}

\subsection{Equivalent versions for finiteness Conjecture}\label{sec;equiv_conj}
In the following, we use $\bN$ to denote the set of non-negative integers and work on $\cR=\bQ(q)$ unless otherwise specified, where $\bQ(q)$ is the field of rational functions in $q$ over $\bQ$.

Let $M$ be an oriented connected compact 3-manifold,
and let $\cC$ be a collection of disjoint embedded closed curves in $\partial M$. 
We use
$M^{\cC}$ to denote the 3-manifold obtained from $M$ by attaching 
2-handles along all the closed curves in  $\cC$, and call the pair 
$(M,\cC)$ an \textbf{attaching system}. Let $\iota\colon M\rightarrow M^{\cC}$ denote the natural embedding and $\iota_\ast \colon \sSq(M)\to \sSq( M^{\cC})$ denote the induced homomorphism.
Let $\sSq(\partial M)_{\cC}$ be the  $\bQ(q)$-subvector space of $\sSq(\partial M)$ linearly spanned by all diagrams in $\partial M$ having no intersection with $\cC$. Clearly 
$\sSq(\partial M)_{\cC}$ is a subalgebra of $\sSq(\partial M)$.

 For any $\bQ(q)$-subalgebra $A$ of $\sSq(\partial M)$ and any nonempty subset $X$ of $\sSq(M)$, define
$$A\cdot X = \{a_1\cdot x_1+\cdots + a_n\cdot x_n\mid n\in\mathbb{N}, \ a_i\in A,\ x_i\in X,\ 1\leq i\leq n \},$$
which is a $\bQ(q)$-subvector space of $\sSq(M)$  (when $n = 0$, we define $a_1\cdot x_1+\cdots + a_n\cdot x_n$ to be $0$). 

For any
non-negative integer $g$, let $H_g$ denote the handlebody of genus $g$.

\begin{conj}\label{stro}
For any attaching system $(M,\cC)$, there exists a finite subset $Y$ of $\sSq(M)$ such that $\iota_\ast(\sSq(\partial M)_{\cC}\cdot Y) = \sSq(M^{\cC})$.
\end{conj}

Let us focus on a special case with handlebodies. 
\begin{conj}\label{str}
For any 
attaching system $(H_g,\cC)$, there exists a finite subset $X$ of $\sSq(H_g)$ such that $\iota_\ast(\pHa\cdot X) = \sSq(\Ha)$.
\end{conj}

Let $\rankp \sSq(M)$ denote the minimum number of generators of $\sSq(M)$ as an $\sSq(\partial M)$-module if the minimum number exists  and $\infty$ otherwise. 
Note that if $\sSq(M)$ is a free $\sSq(\partial M)$-module then $\rankp \sSq(M)$ is equal to the usual rank over $\sSq(\partial M)$. 

\begin{thm}\label{equal}
\begin{enumerate}
    \item  If the attaching system $(H_g,\cC)$ satisfies Conjecture \ref{str},  then $H_g^{\cC}$ satisfies Conjecture \ref{conj} and $$\rankp\sSq(\Ha)\leq |X|.$$
    \item For any attaching system $(M,\cC)$, if $M^{\cC}$ satisfies Conjecture \ref{conj}, then $(M,\cC)$ satisfies Conjecture \ref{stro}.
    \item Conjectures \ref{conj}, \ref{stro} and \ref{str} are equivalent to each other. 
\end{enumerate}
\end{thm}

\begin{proof}
(1) 
It suffices to show $\pHa$ is generated by $\iota_\ast(X)$ over $\sSq(\partial \Ha)$. For any diagram $L$ in $\partial H_g$ having no intersection with $\cC$, and any element $x\in X$, $\iota_\ast(L\cdot x)
= \iota_\ast(L\cup x) = \iota(L)\cup \iota_\ast(x)$, where $\iota(L)\subset \partial\Ha$. Thus
$\iota_\ast(L\cdot x)\in \sSq(\partial \Ha)\cdot \iota_\ast(X)$. Then the definition of 
$\pHa$ and the fact $\iota_\ast(\pHa\cdot X) = \sSq(\Ha)$ imply 
$\sSq(\partial \Ha)\cdot \iota_\ast(X) = \sSq(\Ha)$. 
The required inequality follows from the argument. 

(2) 
From Conjecture \ref{conj}, we know there exists a finite subset $Z$ of $\sSq(M^{\cC})$ such that $\sSq(\partial M^{\cC})\cdot Z = \sSq(M^{\cC})$. Since $\iota_\ast \colon \sSq(M)\to \sSq( M^{\cC})$ is surjective \cite{Prz00}, there exists a finite subset $Y$ of $\sSq(M)$ such that $\iota_\ast(Y) = Z$. 
Note that $\partial M^{\cC}$ is obtained from $\partial M$ in the following way. {For a small open neighborhood $N(\cC)$ of $\cC$}, $\partial M\setminus N(\cC)$ is a surface with 2$|\cC|$ circle boundary components. Then, $\partial M^{\cC}$ is obtained from $\partial M\setminus N(\cC)$ by gluing 2$|\cC|$ disks along all these circle boundary components. Thus any diagrams in $\partial  M^{\cC}$ can be pushed into $\partial  M\setminus N(\cC)$. Then we have $\sSq(M^{\cC})=\sSq(\partial M^{\cC})\cdot Z \subset \iota_*(\sSq(\partial M)_{\cC}\cdot Y)$.

(3) 
From (2), we know Conjecture \ref{conj} implies Conjecture \ref{stro}. Obviously, Conjecture \ref{stro} implies Conjecture \ref{str}.
Note that for any oriented connected compact  3-manifold $M$,  there exists an attaching system $(H_g,\cC)$  such that $M $ (maybe after cutting out some open 3-balls) is isomorphic to $\Ha$, see e.g. \cite[Theorem 3.1.10]{SSS05}. Then from (1), Conjecture \ref{str} implies Conjecture \ref{conj}.
\end{proof}

\begin{prop}\label{prop2.4}
 Suppose Conjecture \ref{stro} holds for an attaching system $(M,\cC)$, and $\varphi$ is any diffeomorphism from $M$ to itself. Then Conjecture \ref{stro} also holds for $(M,\varphi(\cC))$.
\end{prop}
\begin{proof}
For any subset $Y$ of $\sSq(M)$, we have 
$\varphi_*(\sSq(\partial M)_{\cC}\cdot Y)=\sSq(\partial M)_{\varphi(\cC)} \cdot\varphi_*(Y)$. 
This completes the proof.
\end{proof}

\begin{prop}\label{disk}
If $\cC$ consists of closed curves such that all the curves bound embedded disks in $H_g$, then Conjecture \ref{str} holds. 
\end{prop}
\begin{proof}
Any skein in $H_g^{\cC}$ is a linear sum of skeins in $H_g^{\cC}$ having no intersection with the embedded disks bounded by closed curves in $\cC$, see \cite[Theorem 6.3]{Prz00}. Then we can set $X = \{\emptyset\}$, where $\emptyset$ represents the empty skein.
\end{proof}

Let $\alpha$ be an embedded closed curve in a surface $\Sigma$. We say $\alpha$ is {\bf trivial} if it bounds an embedded disk, and say $\alpha$ is {\bf separating} if $\Sigma\setminus\alpha$ has one more component than $\Sigma$.

Define $n(\cC)$ to be the maximum number of non-trivial and non-separating closed curves which are mutually not isotopic in $\cC$.

For any connected surface $\Sigma$, let 
$g(\Sigma)$ denote the sum of genera of all the connected components of $\Sigma$. 

\begin{cor}\label{cor2.6}
Conjecture \ref{str} holds when $g(\partial  H_g)=n(\cC)$ and $\cC$ generates $H_1(H_g)$.
\end{cor}
\begin{proof}
We know
 $H_g^{\cC}$ has the same skein module as a closed 3-manifold, where $H_g^{\cC}$ is a closed $3$-manifold minus $3$-balls. Then Theorem \ref{GJS} and Theorem \ref{equal} complete the proof.
\end{proof}

Let $(M_1,\cC_1),(M_2,\cC_2)$
be two attaching systems. For each $i=1,2$, suppose $D_i$  is an embedded disk in $\partial M_i$ such that $D_i$ has no intersection with $\cC_i$. We use an orientation reversing diffeomorphism to glue $D_1$ and $D_2$. 
For the pair ${\bf D}=(D_1,D_2)$, denote the resulting 3-manifold as $M_1\#_{\bf D}M_2$, and the gluing disk, denoted as $D$, is a properly embedded disk in $M_1\#_{\bf D}M_2$. 
For each $i=1,2$,
suppose $D_i$ is contained in the boundary component $U_i$.
Consider $(M_1\#_{\bf D}M_2,\cC)$ with 
$\cC = \cC_1\cup\cC_2$ if $M_i^{\cC_i}$ has a sphere boundary component containing $D_i$ for $i=1$ or $2$, or
$\cC = \cC_1\cup\cC_2\cup \{\partial D\}$ otherwise.


\begin{prop}
If Conjecture \ref{stro} holds for both $(M_1,\cC_1)$ and $(M_2,\cC_2)$,
then it also holds for $(M_1\#_{\bf D}M_2,\cC)$.
\end{prop}
\begin{proof}
Suppose $M_i^{\cC_i}$ has a sphere boundary component containing $D_i$ for $i=1$ or $2$. 
Without loss of generality, we assume 
that $D_1$ lies in a sphere boundary component $S^2$ of $M_1^{\cC_1}$.
By applying the sphere sliding, explained in Appendix \ref{sec:sliding}, to any skein $\alpha$ in $(M_1\#_{\bf D}M_2)^{\cC}$, we get, in $\sSq((M_1\#_{\bf D}M_2)^{\cC})$,
 $\alpha$ is a linear sum of skeins having no intersection with $D$.
 Then the assumption implies Conjecture \ref{str} holds for $(M_1\#_{\bf D}M_2,\cC)$.

 Suppose $\cC = \cC_1\cup\cC_2\cup \{\partial D\}$. For this case, $\partial D\in\cC$.  The same argument, as above, works by doing handle sliding along the $2$-handle attached to $\partial D$, see also \cite{Prz00}.
\end{proof}

\subsection{Dehn--Thurston coordinates}
 In this subsection, we recall the definition of Dehn--Thurston coordinates, see e.g. \cite{PH92} for more details. 

A \textbf{multicurve} on a surface is a disjoint union of simple closed curves on the surface. 
A multicurve is \textbf{essential} 
if it  has no null-homotopic component.

Let $\Sigma_g$ denote the orientable closed surface of genus $g$ and $\{C_i\}_{i=1}^{3g-3}$ be the set of non-trivial simple closed curves on $\Sigma_g$ such that any two of them are not homotopic, known as a {\bf pants decomposition} of $\Sigma_g$. 
 Take a small closed neighborhood $N(C_i)$ of $C_i$ in $\Sigma_g$ ($i=1,\dots 3g-3$). 
After we remove $\sqcup_{i=1}^{3g-3} \Int N(C_i)$ from $\Sigma_g$, the resulting surface is a disjoint union of pairs of pants, where a pair of pants is a surface homeomorphic to  $S^2$ minus 3 open disks. 
A {\bf dual graph} is a trivalent graph $\Gamma$ on $\Sigma_g$ such that 
$C_i\ (i=1,\dots,3g-3)$ intersects with $\Gamma$ just once and the intersection of $\Gamma$ and each pair of pants has just 1 trivalent vertex.  

For a multicurve $\gamma$ on $\Sigma_g$, let $n_i(\gamma)\ (i=1,2,\dots,3g-3)$ be the geometric intersection number of $C_i$ and $\gamma$, i.e. the minimum of the intersection number of $C_i$ and $\gamma'$ among $\gamma'$ homotopic to $\gamma$. 

For an essential multicurve of $\Sigma_g$, 
it is in \textbf{general position} if each component of the intersection of the multicurve and the pairs of pants of $\Sigma\setminus(\sqcup_{i=1}^{3g-3} \Int N(C_i))$ is one of the curves depicted in Figure \ref{fig;curves_pants}. Here, any simple closed curve parallel to $C_i$ is in $N(C_i)$. We also impose that it intersects with $C_i$ and $\Gamma$ transversely.  
It is known that one can isotope any essential multicurve to be in general position and an essential multicurve $\gamma$ in  general position realizes the geometric intersections, i.e. $n_i(\gamma)=\#\{\gamma\cap C_i\}$. 
We isotope an essential multicurve $\gamma$ to be in general position, denoted by $\gamma'$. 
Let $t_i(\gamma)\ (i=1,\dots,3g-3)$ be the geometric intersection number of $\gamma'\cap N(C_i)$ and $\Gamma\cap N(C_i)$ if $\gamma'\cap N(C_i)$ consists of only loops, and be the oriented intersection number of $\gamma'\cap N(C_i)$ and $\Gamma\cap N(C_i)$ defined by 
summing $+1$ (resp. $-1$) for 
$\begin{array}{c}\includegraphics[scale=0.13]{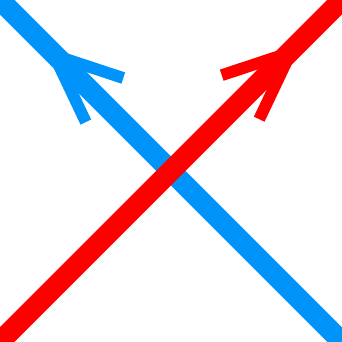}\end{array}$
(resp. $\begin{array}{c}\includegraphics[scale=0.13]{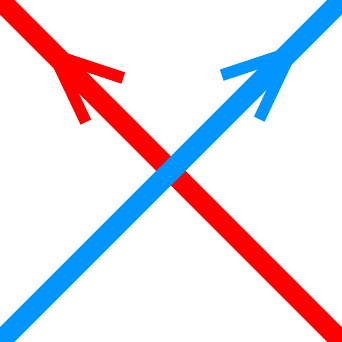}\end{array}$) over all crossings, where the blue curve is $\Gamma\cap N(C_i)$ and the red curve is a part of $\gamma'\cap N(C_i)$ and the orientations of $\Gamma\cap N(C_i)$ and all the arcs of $\gamma'\cap N(C_i)$ are assigned to be inward/outward on the same boundary component of $N(C_i)$. 
It is known that $t_i(\gamma)$ is well-defined. 
The \textbf{Dehn--Thurston coordinate} of an essential multicurve $\gamma$ 
with respect to $\{C_i\}$ and $\Gamma$ is the coordinate $$(n_1(\gamma),\dots,n_{3g-3}(\gamma), t_1(\gamma),\dots,t_{3g-3}(\gamma))\in \bN^{3g-3}\times \bZ^{3g-3}.$$ 
In the following, we abbreviate it to the DT coordinate and let $\DT(\gamma)$ denote the coordinate of $\gamma$. 
It is known that there is a one-to-one correspondence between the set of the DT coordinates of essential multicurves on $\Sigma_{g}$ and the set of isotopy classes of essential multicurves on $\Sigma_g$. 

\begin{figure}[h]
      \centering
      \includegraphics[scale=0.35]{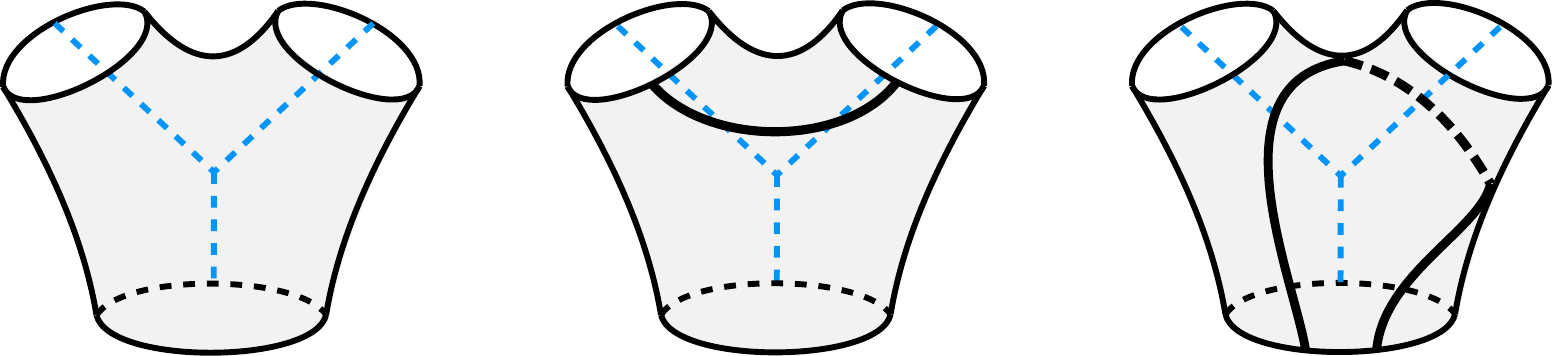}
      \caption{Left: a pair of pants with a part of the dual graph (the blue graph), Middle and Right: each bold arc is an embedded arc in a pair of pants whose endpoints avoid the part of the dual graph. }\label{fig;curves_pants}
  \end{figure}

 We will use the pants decompositions depicted in Figures \ref{Fig;pantsdecomp_genus2} and \ref{fig;edges_pants_g3} for $\partial H_2$ and $\partial H_3$ respectively.
\begin{figure}[ht]\centering\includegraphics[width=130pt]{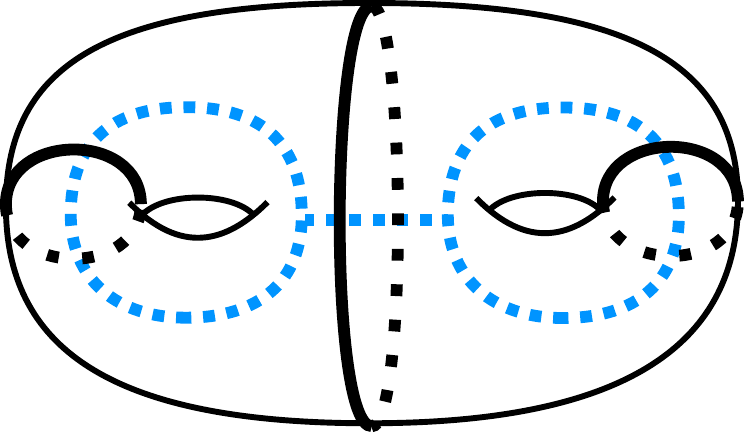}\caption{The bold curves on $\partial H_2$ are $C_1,C_3,C_2$ from left to right and the blue graph is a dual graph.}\label{Fig;pantsdecomp_genus2}\end{figure}

\subsection{Generating sets and bases for $H_2$ and $H_3$} \label{sec;gene_bases} 

For any non-negative integer $m$, let $\Sigma_0^{m}$ denote the surface obtained from $S^2$ by removing $m$ open disks. Then $H_g$ can be regarded as $\Sigma_{0}^{g+1}\times [0,1]$.

Let $x,y,z$ be the peripheral loops of $\Sigma_0^3$ depicted as in Figure \ref{fg_xyz}. 
We also use $x,y,z$ to denote diagrams on $\Sigma_0^{3}\times\{1\}\subset \partial (\Sigma_0^{3}\times [0,1])=\partial H_{2}$. So we can also regard $x,y,z$ as elements in $\sSq(\partial H_2)$.

\begin{figure}[h]
      \centering
      \includegraphics[scale=0.4]{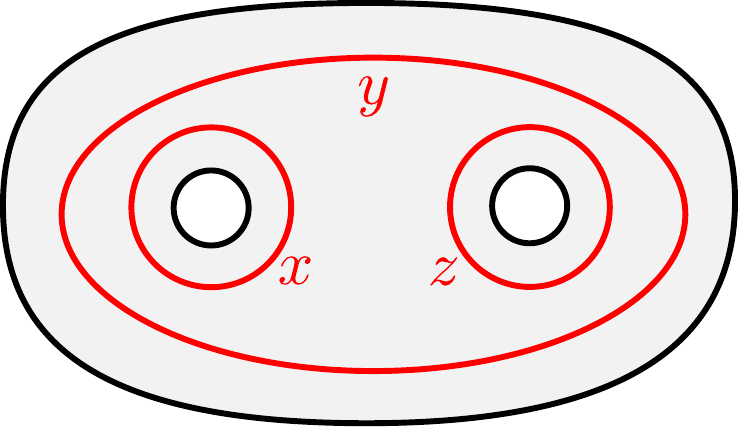}
      \caption{$\Sigma_0^3$ and algebraic generators $x,y,z$ in $\Sigma_0^3$}\label{fg_xyz}
  \end{figure}

The ground ring $\mathcal R=\mathbb Z[q^{\pm 1}]$ (the Laurent polynomial ring in $q$ with integer coefficients) in Lemmas \ref{lem-skeinalgebra}--\ref{relation}.
The following result is well-known. 

\begin{lem}\label{lem-skeinalgebra}For any 
The skein algebra $\sSq(\Sigma_0^{3})$ is isomorphic to $\mathcal{R}[x,y,z]$.
\end{lem}

For a polynomial $P(x,z,y)=\sum_{i,j,k}c(i,j,k)x^i z^j y^k\in\bQ(q)[x,z,y],\ c(i,j,k)\in \bQ(q)$, its {\bf $(x,z;y)$-degree} is the maximum number among $i+j+2k$ with $c(i,j,k)\neq 0$. 

Let $s_1,s_2,s_3,s_{12}, s_{13}, s_{23}, s_{123}$ be the simple closed curves in $\Sigma_0^{4}$ as depicted in Figure \ref{generate1}.
\begin{figure}[h]
      \centering
      \includegraphics[scale=0.3]{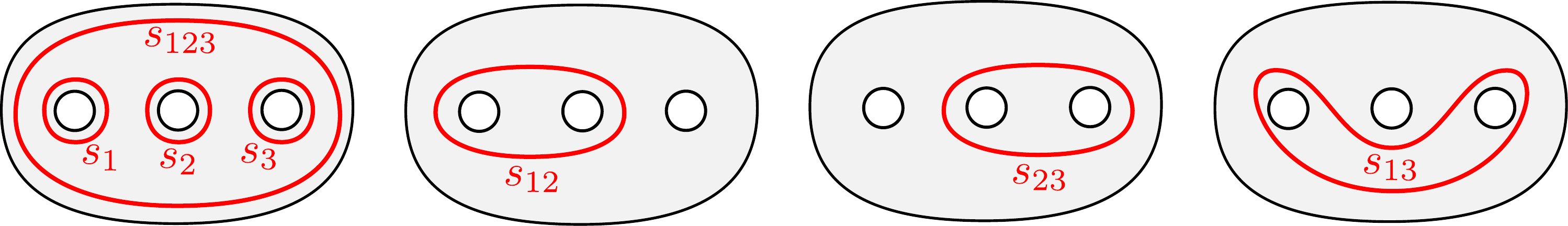}
      \caption{ The algebraic generators $s_1,s_2,s_3,s_{12}, s_{13}, s_{23}, s_{123}$ of $\sSq(\Sigma_{0}^4)$.}\label{generate1}
  \end{figure}

For later convenience, for any $\vec{k}=(k_1,k_2,k_3,k_{12},k_{13},k_{23},k_{123})\in\mathbb{N}^7$, put \begin{align}
s^{\vec{k}}&:=s_1^{k_1}s_2^{k_2}s_3^{k_3} s_{12}^{k_{12}} s_{13} ^{k_{13}}s_{23}^{k_{23}}s_{123}^{k_{123}}\in \sSq(\Sigma_0^{4}),\nonumber\\ 
\text{s}(\vec{k})&:=3(k_1+k_2+k_3+k_{123}) + 4(k_{12}+k_{13}+k_{23}),\nonumber\\
\text{s}'(\vec{k})&:=
2(k_1+k_2)+k_3 +2k_{12} + 3(k_{13}+k_{23} + k_{123})
\label{s(vec)}.
\end{align}

Let $\Lambda$ denote the subset of $\mathbb{N}^7$ defined by 
$$\Lambda = \{(k_1,k_2,k_3,k_{12},k_{13},k_{23},k_{123})\in\mathbb{N}^7\mid k_{12}k_{13}k_{23} = 0\}.$$

\begin{lem}[\cite{BP00}]\label{lem;basis}
As a free module over $\mathcal{R}$, the skein algebra $\sSq(\Sigma_0^{4})$ has the  basis 
$\{s^{\vec{k}}\mid \vec{k}\in\Lambda\}.$
\end{lem}

\begin{lem}[\cite{BP00}]\label{relation}
When $q = 1$, the skein algebra $\sS_1(\Sigma_0^{4})$ is a commutative algebra over $\mathbb{Z}$ generated by $s_1,s_2,s_3,s_{12},s_{13},s_{23},s_{123}$ subject to the following relation:
\begin{align*}
    s_{12}s_{13}s_{23} =& s_{12}^2 + s_{13}^2 + s_{23}^2 +s_{12}(s_1s_2+s_3s_{123}) + s_{13}(s_1s_3+s_2s_{123}) + s_{23}(s_2s_3+s_1s_{123})\\&
    +s_1s_2s_3s_{123} + s_1^{2} + s_2^{2} + s_3^{2}+s_{123}^{2}-4.
\end{align*}
\end{lem}

\begin{figure}[h]
      \centering
      \includegraphics[scale=0.38]{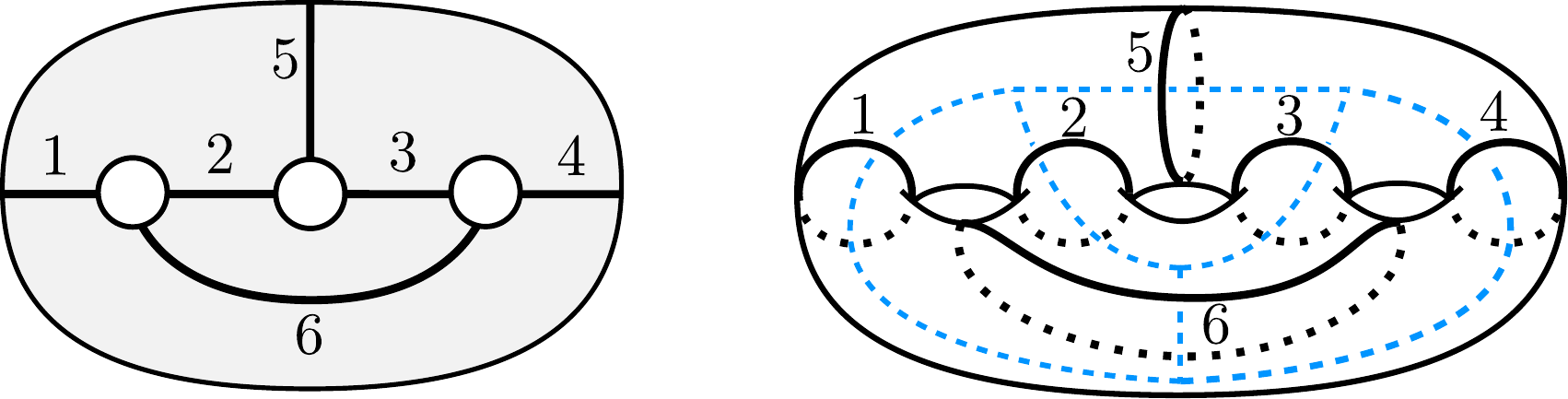}
      \caption{Left: $\Sigma_0^{4}$ with edges giving a triangulation of $\text{Int}\Sigma_0^{4}$. The edges are numbered as in the picture. Right: a pants decomposition $\{C_i\}_{i=1}^6$ for $\Sigma_3^{0}$ and a dual graph (the blue graph). Each $C_i$ corresponds with the $i$-th curve.}\label{fig;edges_pants_g3}
  \end{figure}

For an essential multicurve $\gamma$ in $\Sigma_0^4$, let $m_i(\gamma)$ denote the geometric intersection number of $\gamma$ and the $i$-th edge in Figure \ref{fig;edges_pants_g3}.
Put 
\begin{align}
\vec{m}(\gamma) := 
(m_1(\gamma),m_2(\gamma),m_3(\gamma),m_4(\gamma),m_5(\gamma),m_6(\gamma))\in \bN^6\label{vec(m)}\\
{\rm sum}(\gamma) := m_1(\gamma)+m_2(\gamma)+m_3(\gamma)+m_4(\gamma)+m_5(\gamma)+m_6(\gamma).\label{sum}
\end{align}
Then we have
\begin{align*}
&\vec{m}(s_1)
= (1,1,0,0,0,1),\ \vec{m}(s_2) = (0,1,1,0,1,0),\ 
\vec{m}(s_3)
= (0,0,1,1,0,1),\\ &\vec{m}(s_{12}) = (1,0,1,0,1,1),\ 
\vec{m}(s_{13})
= (1,1,1,1,0,0),\ \vec{m}(s_{23}) = (0,1,0,1,1,1),\\
&\vec{m}(s_{123}) = (1,0,0,1,1,0).
\end{align*}
For any $\vec{k}\in\mathbb{N}^7$, it is easy to show that 
$\text{s}(\vec{k})$ is the sum of all the entries of 
\begin{align}
k_1 \vec{m}(s_1)+k_2 \vec{m}(s_2)+k_3 \vec{m}(s_3)+k_{12} \vec{m}(s_{12})+k_{13} \vec{m}(s_{13})+k_{23} \vec{m}(s_{23})+k_{123}\vec{m}(s_{123}).\label{value}
\end{align}

\begin{lem}\label{inde}
For any distinct $\vec{u},\vec{v}\in\Lambda$, 
the values (\ref{value}) for $\vec{u}$ and $\vec{v}$ are also distinct. 
\end{lem}

\begin{proof}
One can show that, by regarding 
$$\vec{m}(s_1),\vec{m}(s_2),\vec{m}(s_3),\vec{m}(s_{12}),\vec{m}(s_{13}),\vec{m}(s_{23}),\vec{m}(s_{123})$$
as vectors and removing one of $\vec{m}(s_{12}),\vec{m}(s_{13}), \vec{m}(s_{23})$ from them, 
the remaining vectors are linearly independent,
and  
\begin{align}
\vec{m}(s_1)+\vec{m}(s_2)+\vec{m}(s_3)-\vec{m}(s_{12})-\vec{m}(s_{13})-\vec{m}(s_{23})+\vec{m}(s_{123}) = 0.\label{eq;linear_eq}
\end{align}
Suppose 
$$k_1 \vec{m}(s_1)+k_2 \vec{m}(s_2)+k_3 \vec{m}(s_3)+k_{12} \vec{m}(s_{12})+k_{13} \vec{m}(s_{13})+k_{23} \vec{m}(s_{23})+k_{123}\vec{m}(s_{123}) = 0$$
for $(k_1,k_2,k_3,k_{12},k_{13},k_{23},k_{123})\in\mathbb Q^{7}.$
The above discussion shows the solution space of \eqref{eq;linear_eq} is $1$-dimensional, i.e. 
$$(k_1,k_2,k_3,k_{12},k_{13},k_{23},k_{123}) = c(1,1,1,-1,-1,-1,1),$$
where $c\in\mathbb Q$.
Thus 
\begin{align}\label{eq-zero}
    k_{12},k_{13},k_{23}\leq 0\text{ or }k_{12},k_{13},k_{23}\geq 0.
\end{align}

Assume there exist distinct $\vec{u},\vec{v}\in\Lambda$ such that 
the values (\ref{value}) of $\vec{u}$ and $\vec{v}$ are the same.
Then we have 
\begin{align}\label{eq-zero-uv}
(u_1-v_1 )\vec{m}(s_1)+(u_2-v_2) \vec{m}(s_2)+(u_3-v_3) \vec{m}(s_3)+(u_{12}-v_{12})k_{12} \vec{m}(s_{12})+\nonumber\\
(u_{13}-k_{13}) \vec{m}(s_{13})+(u_{23}-v_{23})\vec{m}(s_{23})+(u_{123}-v_{123})\vec{m}(s_{123})=0
\end{align}
If $(u_{12}-v_{12})(u_{13}-v_{13})(u_{23}-v_{23})=0$, then linear independence discussed above shows $\vec{u}=\vec{v}$. This contradicts with the assumption that $\vec{u}\neq \vec{v}$.
Suppose $(u_{12}-v_{12})(u_{13}-k_{13})(u_{23}-v_{23})\leq 0$. 
Since $\vec u\in\Lambda$, i.e. $u_{12}u_{13}u_{23}=0$, 
we suppose $u_{12}=0$ without loss of generality. 
Since $u_{12}-v_{12}\neq 0$, we have $v_{12}\neq 0$ and $u_{12}-v_{12}<0$. Then $v_{13}=0$ or $v_{23}=0$ from $\vec v\in\Lambda$. We suppose $v_{13}=0$.  Since $u_{13}-v_{13}\neq 0$, we have $u_{13}\neq 0$ and $u_{13}-v_{13}>0$.
By combining \eqref{eq-zero-uv}, $u_{12}-v_{12}<0$ and $u_{13}-v_{13}>0$, the assumption contradicts with \eqref{eq-zero}.
\end{proof}

\begin{rmk}
     Lemma \ref{inde}  implies the independence of the basis elements in Lemma  \ref{lem;basis}.
\end{rmk}

\subsection{Proof of finiteness conjecture for a family of 3-manifolds}
In this subsection, the goal is to prove the following theorem. 

\begin{thm}\label{thm;example}
 Let $\gamma\subset \partial H_g$ be a simple closed curve and $H_g^\gamma$ denote the resulting 3-manifold obtained from $H_g$ by attaching a 2-handle along $\gamma$. 
If $\gamma$ satisfies one of the following conditions, the finiteness conjecture holds for $H_g^\gamma$, 
where the DT coordinates are with respect to the pants decompositions and the dual graphs defined in Figures \ref{Fig;pantsdecomp_genus2} and \ref{fig;edges_pants_g3}.\\
{\bf Case 1.} $g=2$ with one of the following;
\begin{itemize}
    \item $\DT(\gamma)=(n,n,2n,t_1,t_1,t_2)$ where $t_1t_2\geq 0$ and $n=|2t_2+t_1|$ (symmetric),
    \item $\DT(\gamma)=(n+m,n,2n,t,0,0)$ or $(n,n+m,2n,t,0,0)$ where $t\in \{n,-n\}$,
    \item $\DT(\gamma)=(n_1,n_2,2, \pm1, \pm 1, 0)$, 
\end{itemize}
where $n\in \bZ_{\geq 1},\ m,n_1,n_2\in \bN$.\\
{\bf Case 2.} $g = 3$ with one of the following; 
\begin{itemize}
    \item $\DT(\gamma)=(1,n,n+m,m+1,n+1,m,0,0,0,1,1,0)$ or $(1,n,n+m,m+1,n+1,m,0,0,0,-1,-1,0)$,
    \item $\DT(\gamma)=(n+m,m,0,n,n,m+2n,t,0,0,0,0,0)$,
\end{itemize}
where $n,m\in\bZ_{\geq 1}$.
In particular, $\rankp(H_2^\gamma)=1$ in the second case of Case 1 and $\rankp(H_3^\gamma)=1$ for both two cases in Case 2. 
\end{thm}

For convenience, we will use two different framings for curves on $\partial H_g$ in the proof.
The framings of all gluing  curves (drawn in red) on the boundary in Theorem \ref{thm;example}
are the usual blackboard framings. 
On the other hand, other boundary curves in Theorem \ref{thm;example}
are equipped with the blackboard framings with respect to the projection $H_g=\Sigma_{0}^{g+1}\times [0,1]\to \Sigma_{0}^{g+1}$.

\begin{proof}
{\bf The first case in Case 1.} We can assume both $t_1$ and $t_2$ are non-negative integers since the proof in the case that they are both non-positive integers is similar.
The red curve in the left of Figure \ref{fg1} is the gluing curve $\gamma$. 
{Note that the green curve $\alpha_1$ and the blue curve $\alpha_2$ respectively correspond to $x$ and $z$ in $\sSq(\Sigma_{0}^{3})(=\sSq(H_2))$ depicted in the left of Figure \ref{fg1}. 
Since the green curve and the blue curve do not intersect with the red curve, $x,z\in \sSq(\partial H_g)_\gamma$.}
\begin{figure}[h]
      \centering
      \includegraphics[scale=0.45]{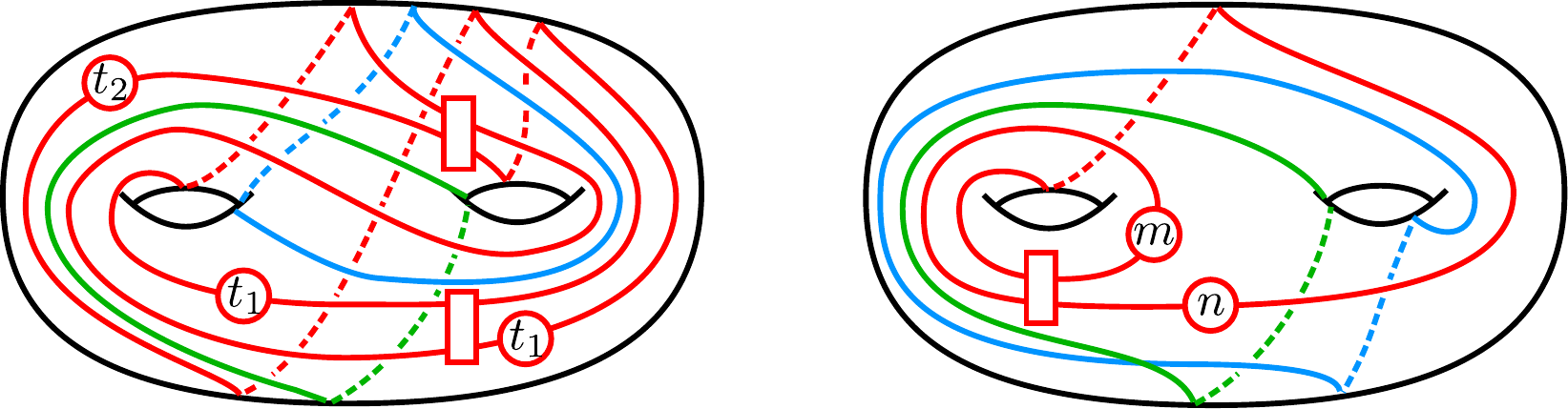}
      \caption{Each circle with non-negative integer $k$ represents $k$ parallel copies of the corresponding arcs. Each coupon in the picture represents certain number of horizontal lines connecting the endpoints on the left side of the coupon and those on the right side of the coupon. \\
      Left: the red curve is the gluing curve $\gamma$ whose DT coordinate is $(n,n,2n,t_1,t_1,t_2)$. Right: the red curve is the gluing curve $\gamma$ whose DT coordinate is $(n+m,n,2n,n,0,0).$}\label{fg1}
  \end{figure}
Obviously we have $\gamma \in \sSq(\partial H_2)_{\gamma}$. 
Set $X = \{1,y,\cdots,y^{n-1}\}\subset \sSq(H_2)$.
Then it suffices to show  $\sSq(\partial H_2)_{\gamma}\cdot X
= {\sSq(\Sigma_0^3)(=\sSq(H_2))}$ as $\mathbb{Q}(q)$-vector spaces. 
As we mentioned, $\alpha_1=x,\ \alpha_2=z,\ y\in \sSq(\partial H_2)_{\gamma}\cdot X$.
Fix $2\leq k\leq 2n-1$ and assume that $x^{k_1} z^{k_2} y^{k_3}\in \sSq(\partial H_2)_{\gamma}\cdot X$ for any solutions of $k_1+k_2+2k_3<k$ $(k_1,k_2,k_3\in \mathbb{N})$. For $k_1,k_2,k_3\in \mathbb{N}$ satisfying $k=k_1+k_2+2k_3$, consider $\alpha_1^{k_1}\alpha_2^{k_2}\cdot y^{k_3}$. 
From the geometric intersection numbers in the DT coordinates of $\alpha_1, \alpha_2$ and $y$, 
\begin{align}
\alpha_1^{k_1}\alpha_2^{k_2}\cdot y^{k_3}
=\sum_{k=i+j+2l}c(k_1,k_2,k_3;i,j,l)x^i z^j y^l+\text{(lower $(x,z;y)$-degree terms)}, \label{eq;xzy00}
\end{align}
where $c(k_1,k_2,k_3;i,j,l)\in \bZ[q^{\pm1}]\subset \bQ(q)$. 
When we substitute $1$ to $q$, $\sS_1(H_2)$ is a commutative algebra and we have $\alpha_1^{k_1}\alpha_2^{k_2}\cdot y^{k_3}=x^{k_1} z^{k_2} y^{k_3}\in \sS_1(H_2)$. With the lexicographic order on $(k_1,k_2,k_3)$, 
the coefficient matrix $(c(k_1,k_2,k_3;i,j,l))$ obtained from the highest terms in (\ref{eq;xzy00}) is full rank since it will be the identity matrix by substituting $1$ to $q$, 
where $(c(k_1,k_2,k_3;i,j,l))$ is a square matrix whose rows and columns run over all solutions of $k=k_1+k_2+2k_3$ and $k=i+j+2l$ respectively. 
This implies that $x^{k_1} z^{k_2} y^{k_3}\in \sSq(\partial H_2)_{\gamma}\cdot X.$

Fix $k\geq 2n$. Assume that $x^{k_1} z^{k_2} y^{k_3}\in \sSq(\partial H_2)_{\gamma}\cdot X$ for any $k_1,k_2,k_3\in \bN$ satisfying $k_1+k_2+2k_3<k$.  Consider $\alpha_1^{k_1}\alpha_2^{k_2}\cdot y^{k_3}$ for $k_1,k_2,k_3\in \bN$ satisfying $k=k_1+k_2+2k_3$ and $k_3< n$. 
Also consider $\gamma^{m}\alpha_1^{k_1}\alpha_2^{k_2}\cdot y^{k_3-mn}$ for $k_1,k_2,k_3,m\in \bN$ satisfying $k=k_1+k_2+2k_3$, $k_3\geq n$ and $0\leq k_3-mn<n$. 
When we substitute $1$ to $q$ then we take the absolute value of coefficients,
we have the (absolute valued) coefficient matrix
$\begin{pmatrix}
I &O\\
\ast &T
\end{pmatrix}$, where $I$ is the identity matrix and $T$ is obtained from a triangular matrix whose diagonal entries are $1$ by interchanging some rows. Since the matrix is full rank, $x^{k_1} z^{k_2} y^{k_3}\in \sSq(\partial H_2)_{\gamma}\cdot X.$

\noindent{\bf The second case in Case 1.} The red curve in the right of Figure \ref{fg1} is  the gluing curve with the DT coordinate $(n+m,n,2n,n,0,0)$. We only prove this case since a similar discussion works for other parallel cases. The green curve $\alpha$ in the right of Figure \ref{fg1} shows $x\in\sSq(\partial H_2)_{\gamma}$.
The blue curve $\beta$ in the right of Figure \ref{fg1} shows $y\in\sSq(\partial H_2)_{\gamma}$. Obviously we also have $z\in\sSq(\partial H_2)_{\gamma}$. We set $X = \{\emptyset\}$. 
For any $k_1,k_2,k_3\in \bN$, we have $$z^{k_2}\beta^{k_3}\alpha^{k_1}\cdot \emptyset = x^{k_1}z^{k_2}y^{k_3}\in \sSq(\partial H_2)_{\gamma}\cdot X.$$
In particular, $\rankp(H_2^\gamma)=1$ in this case.

\begin{figure}[h]
      \centering
      \includegraphics[scale=0.45]{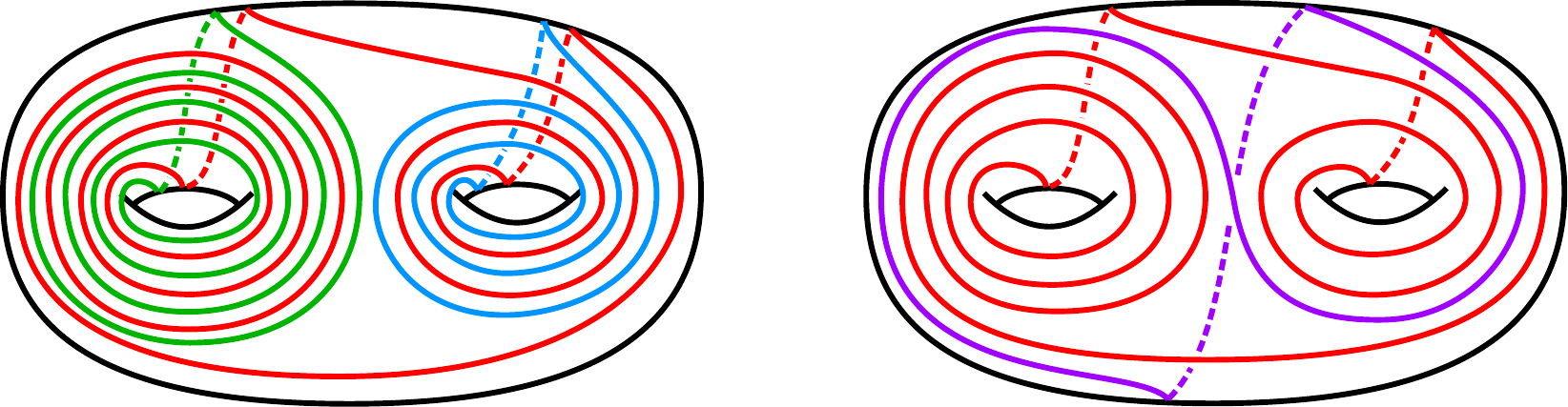}
      \caption{The red curve is the gluing curve $\gamma$ with $\DT(\gamma)=(4,3,2,1,1,0)$.}\label{fg3}
  \end{figure}
\noindent{\bf The third case in Case 1.} We only prove $\DT(\gamma)=(n_1,n_2,2, 1, 1, 0)$ since a similar argument works for other parallel cases. 
For readers convenience, the red curve in Figure \ref{fg3} is the gluing curve with $n_1=4,n_2=3$. 
If $n_1 = 0$ or $n_2 =0$, it is covered by Corollary \ref{cor2.6}. Suppose $n_1 n_2\neq 0$. 
Let $\alpha_1, \alpha_2$ and $\alpha_3$ denote simple closed curves on $\partial H_2$ whose DT coordinates are $(n_1,0,0, 1, 0, 0),\ (0,n_2,0, 0, 1, 0)$, and $(1,1,2, 0, 0, 1)$ respectively. 
Note that $\alpha_i$ does not intersect with $\gamma$, i.e. $\alpha_i\in\sSq(\partial H_2)_{\gamma}$ for $i=1,2,3$. 
In $\sSq(H_2)$, we have 
$$
\alpha_1 = u_{n_1} x^{n_1}
+ \cdots +u_1 x+u_0,\quad 
\alpha_2 = v_{n_2} z^{n_2}
+ \cdots +v_1 z+v_0,\quad 
\alpha_3 = qy+q^{-1}xz,
$$ where $u_{n_1}$ and $v_{n_2}$ are nonzero elements in $\bQ(q)$.
Set 
\begin{align}
X = \{x^{k_1}z^{k_2}\mid k_i\in\mathbb{N},\ k_i\leq n_i-1 \text{ for }i=1,2\}. \label{def;X03}
\end{align}
To show 
\begin{align}
\{x^{k_1}z^{k_2}\mid k_1,k_2\in\mathbb{N}\}\subset \sSq(\partial H_2)_{\gamma}\cdot X,\label{eq;sub03}
\end{align}
assume that, for $k\geq1$, $x^{k_1}z^{k_2}\in \sSq(\partial H_2)_{\gamma}\cdot X$ for any $k_1,k_2\in \bN$ with $k_1+k_2<k$.
For $k_1,k_2\in \bN$ with $k_1+k_2=k$, 
consider $\alpha_1^{m_1}\alpha_2^{m_2}x^{k_1-m_1n_1}z^{k_2-m_2 n_2}$, where $m_i\in \bN$ satisfies $0\leq k_i-m_in_i<n_i\ (i=1,2)$. 
Then, we have
$$\alpha_1^{m_1}\alpha_2^{m_2}\cdot x^{k_1-m_1n_1}z^{k_2-m_2 n_2}=q^{c}x^{k_1}z^{k_2}+(\text{lower $(x,z;y)$-degree terms}), $$
where $c\in \bZ$ and $y$ does not appear on the right-hand side.  
Hence, $x^{k_1}z^{k_2}\in \sSq(\partial H_2)_{\gamma}\cdot X$, i.e. (\ref{eq;sub03}) holds. 

Fix $k\geq1$ and suppose that $x^{k_1}z^{k_2}y^{k_3}\in \sSq(\partial H_2)_{\gamma}\cdot X$ for any $k_1,k_2\in \bN,\ k_3<k$. 
Then, we have
$$\alpha_3^k \cdot x^{k_1}z^{k_2}=q^{c}x^{k_1}z^{k_2}y^{k}+\text{(lower $y$-degree terms)}, $$
where $c\in \bZ$. 
This implies that $x^{k_1}z^{k_2}y^{k}\in \sSq(\partial H_2)_{\gamma}\cdot X.$

    
\noindent{\bf The first case in Case $2$.} 
We only prove the case when $\DT(\gamma)=(1,n,n+m,m+1,n+1,m,0,0,0,1,1,0).$
We will use the elements $\alpha_1,\alpha_2,\alpha_3,\alpha_{12},\alpha_{13},\alpha_{23},\alpha_{123}\in \sSq(\partial H_3)_{\gamma}$ which are shown in Figures \ref{alpha1},\ref{alpha2},\ref{alpha3}.

We set $X = \{\emptyset\}$. For any $k_1,k_2,k_3\in\mathbb{N}$, we have 
$\alpha_1^{k_1}\alpha_2^{k_2}\alpha_3^{k_3}\cdot \emptyset = s_1^{k_1}s_2^{k_2}s_3^{k_3}\in \sSq(\partial H_3)_{\gamma}\cdot X.$
\begin{figure}[h]
      \centering
      \includegraphics[scale=0.45]{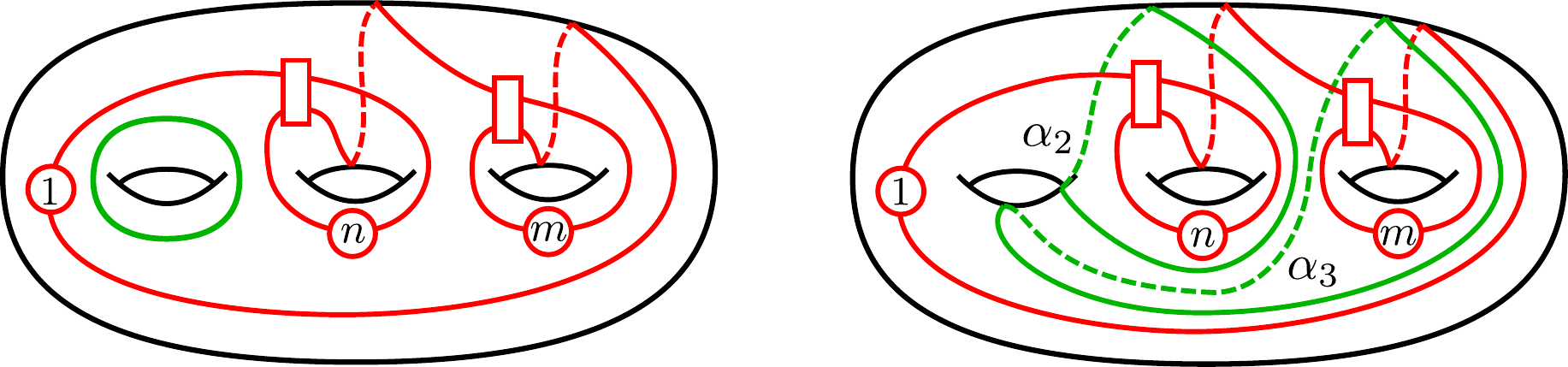}
      \caption{The red curve is the gluing curve $\gamma$. The  green curve in the left picture is $\alpha_1$ and the green curves in the right picture are $\alpha_{2}$ and $\alpha_{3}$.}\label{alpha1}
  \end{figure}

\begin{figure}[h]
      \centering
      \includegraphics[scale=0.45]{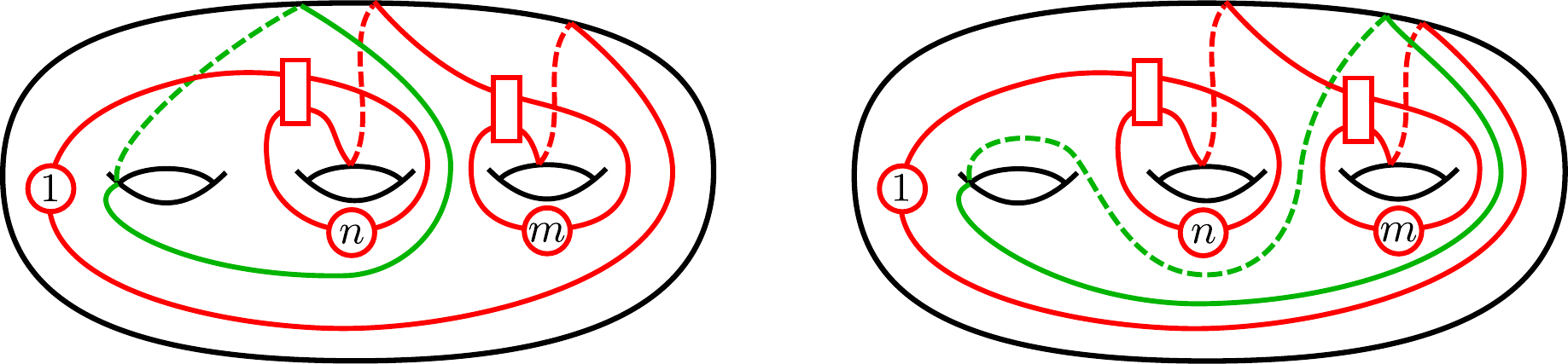}
      \caption{The  green curve in the left (resp. right) picture is $\alpha_{12}$ (resp. $\alpha_{23}$).}\label{alpha2}
  \end{figure}

  \begin{figure}[h]
      \centering
      \includegraphics[scale=0.45]{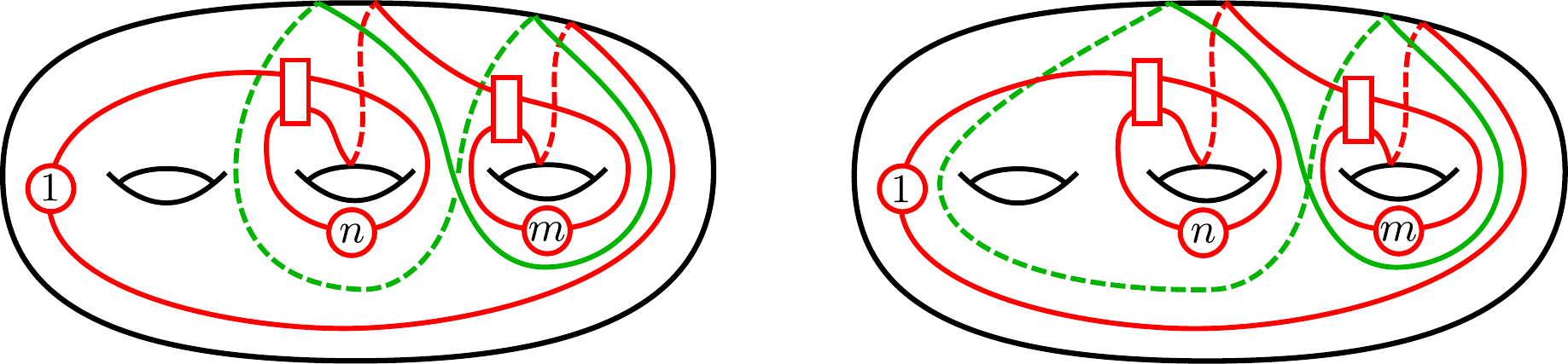}
      \caption{The  green curve in the left (resp. right) picture is $\alpha_{23}$ (resp. $\alpha_{123}$).}\label{alpha3}
  \end{figure}

For any $\vec{k}\in\Lambda$, we define 
$$\deg(s^{\vec{k}}) := \text{s}'(\vec{k}),$$ 
where ${\rm s}'(\vec{k})$ was defined as (\ref{s(vec)}).
We will use mathematical induction on $\deg(s^{\vec{k}})$ to show $s^{\vec{k}}\in \sSq(\partial H_3)_{\gamma}\cdot X$ for any $\vec{k}\in\Lambda$.

We will use Lemmas \ref{keylem}
 and \ref{lem;prod_sum_sss} in the rest of the proof, and will prove them in Section \ref{sec;pf_lem}. 
\begin{lem}\label{keylem}
For any $\vec{k}=(k_1,k_2,k_3,k_{12},k_{13},k_{23},k_{123})\in \mathbb{N}^7$, suppose there exists a finite subset $\Lambda_0$ of $\Lambda$ 
such that 
\begin{equation}\label{eqkey}
\alpha_{123}^{k_{123}}\alpha_{12}^{k_{12}}
\alpha_{13}^{k_{13}}\alpha_{23}^{k_{23}}\cdot(s_1^{k_1}s_2^{k_2}s_3^{k_3}) = 
\sum_{\vec{u}\in\Lambda_0} C_{\vec{u}}s^{\vec{u}},
\end{equation}
where $0\neq C_{\vec{u}}\in \bZ[q^{\pm1}]$ for any $\vec{u}\in\Lambda_0$.
Then, we have  $\deg(s^{\vec{u}})\leq \text{s}'(\vec{k})$  
for any $\vec{u}\in\Lambda_0$. 
\end{lem}

\begin{lem}\label{lem;prod_sum_sss}
    Suppose $q=1$. For any positive integers $k_{12},k_{13},k_{23}$, we have 
    $$s_{12}^{k_{12}} s_{13}^{k_{13}} s_{23}^{k_{23}} =\sum_{\vec{u}\in\Lambda_0} C_{\vec{u}}s^{\vec{u}} + \text{{\rm (terms with degrees less than $2k_{12} + 3(k_{13}+k_{23})$)}},$$
    where $\Lambda_0$ is a finite subset of $\Lambda$ such that, for any $\vec{u}=(u_1,u_2,u_3,u_{12},u_{13},u_{23},u_{123})\in\Lambda_0$, $\deg(\vec{u}) = 2k_{12} + 3(k_{13}+k_{23})$ and $u_{13}u_{23}<k_{13}k_{23}$. 
\end{lem}

Assume that $s^{\vec{k}}\in \sSq(\partial H_3)_{\gamma}\cdot X$ for any $\vec{k}\in \Lambda$ with $\deg(s^{\vec{k}})<k$. 
For any $\vec{v} =(v_1,v_2,v_3,v_{12},v_{13},v_{23},v_{123})\in\Lambda$ with $\deg(s^{\vec{v}})=k$,
Lemma \ref{keylem} implies that 
\begin{align}
\alpha_{123}^{v_{123}}\alpha_{12}^{v_{12}}
\alpha_{13}^{v_{13}}\alpha_{23}^{v_{23}}\cdot(s_1^{v_1}s_2^{v_2}s_3^{v_3}) = 
\sum_{\vec{u}\in\Lambda,\deg(s^{\vec{u}})=k  } C_{\vec{v},\vec{u}}s^{\vec{u}} + \text{(lower degree terms)}.\label{eq;prod_sum_action}
\end{align}

We consider a total order on $\mathbb{N}^7$ defined by the lexicographic order with respect to 
$$(v_{13}v_{23},v_{123},v_{23},v_{13},v_{12},v_{3},v_{2},v_{1}),$$
where $(v_1,v_2,v_3,v_{12},v_{13},v_{23},v_{123})\in\bN^7$. 
Then this order induces the total order on 
$\{s^{\vec{u}}\mid \vec{u}\in\Lambda,\ \deg(s^{\vec{u}}) = k\}$. 
When $q=1$, note that the left-hand side of (\ref{eq;prod_sum_action}) is equal to \begin{align}
&(s_{123}+s_{12}s_3)^{v_{123}}s_{12}^{v_{12}}
s_{13}^{v_{13}}(s_{23}+s_2s_3)^{v_{23}}s_1^{v_1}s_2^{v_2}s_3^{v_3}\nonumber\\
&=\sum_{i=0}^{v_{123}}\sum_{j=0}^{v_{23}}\dbinom{v_{123}}{i}\dbinom{v_{23}}{j} s_{123}^{v_{123}-i}\, s_{12}^{v_{12}+i}\, s_{13}^{v_{13}}\, s_{23}^{v_{23}-j}\, s_1^{v_1}\, s_2^{v_2+j}\, s_3^{v_3+i+j},\label{prod;expand}
\end{align}
where the coefficients denote binomial coefficients.

In the case when $v_{12}=0$ and $v_{13}v_{23}\neq 0$, 
we have 
\begin{align*}
    &\sum_{i=0}^{v_{123}}\sum_{j=0}^{v_{23}}\dbinom{v_{123}}{i}\dbinom{v_{23}}{j} s_{123}^{v_{123}-i}\, s_{12}^{v_{12}+i}\, s_{13}^{v_{13}}\, s_{23}^{v_{23}-j}\, s_1^{v_1}\, s_2^{v_2+j}\, s_3^{v_3+i+j}\\
    =&\sum_{j=0}^{v_{23}}\dbinom{v_{23}}{j} s_{123}^{v_{123}}\, s_{13}^{v_{13}}\, s_{23}^{v_{23}-j}\, s_1^{v_1}\, s_2^{v_2+j}\, s_3^{v_3+i+j}\\
    &+\sum_{i=1}^{v_{123}}\sum_{j=0}^{v_{23}}\dbinom{v_{123}}{i}\dbinom{v_{23}}{j} s_{123}^{v_{123}-i}\, s_{12}^{i}\, s_{13}^{v_{13}}\, s_{23}^{v_{23}-j}\, s_1^{v_1}\, s_2^{v_2+j}\, s_3^{v_3+i+j}.
\end{align*}
The highest term of
$$\sum_{j=0}^{v_{23}}\dbinom{v_{23}}{j} s_{123}^{v_{123}}\, s_{13}^{v_{13}}\, s_{23}^{v_{23}-j}\, s_1^{v_1}\, s_2^{v_2+j}\, s_3^{v_3+i+j}$$
with respect to the total order is $s_{123}^{v_{123}} s_{13}^{v_{13}} s_{23}^{v_{23}} s_1^{v_1} s_2^{v_2} s_3^{v_3}$.
Lemma \ref{lem;prod_sum_sss} implies
the highest term of
$$\sum_{i=1}^{v_{123}}\sum_{j=0}^{v_{23}}\dbinom{v_{123}}{i}\dbinom{v_{23}}{j} s_{123}^{v_{123}-i}\, s_{12}^{i}\, s_{13}^{v_{13}}\, s_{23}^{v_{23}-j}\, s_1^{v_1}\, s_2^{v_2+j}\, s_3^{v_3+i+j}$$
with respect to the total order is lower than 
$s_{123}^{v_{123}} s_{13}^{v_{13}} s_{23}^{v_{23}} s_1^{v_1} s_2^{v_2} s_3^{v_3}$.
Thus the highest term of  (\ref{prod;expand}) with respect to the total order is $s_{123}^{v_{123}} s_{13}^{v_{13}} s_{23}^{v_{23}} s_1^{v_1} s_2^{v_2} s_3^{v_3}$.

In the case when $v_{13}v_{23}=0$, it is easy to see that the highest term of (\ref{prod;expand}) with respect to the total order is $s_{123}^{v_{123}} s_{12}^{v_{12}} s_{13}^{v_{13}} s_{23}^{v_{23}} s_1^{v_1} s_2^{v_2} s_3^{v_3}$. 

The above two cases imply the matrix $(C_{\vec{v},\vec{u}})_{\vec{v},\vec{u}\in\Lambda, \deg(s^{\vec{v}})=\deg(s^{\vec{u}})=k}$ is a lower triangular matrix whose diagonal entries are $1$ when $q=1$.
Then, using the same technique as in the proof of the first case of Case 1, one can show 
$s^{\vec{v}}\in \sSq(\partial H_3)_{\gamma}\cdot X$ for any $\vec{v}\in\Lambda$. Thus $\iota_{*}(\sSq(\partial H_3)_{\gamma}\cdot X) = \sSq(H_3^{\gamma})$.

\noindent{\bf The second case in Case $2$.} 
In the remaining of the proof, we will use $\vec{k}\in \bN^7$ to denote $(k_1,k_2,k_3,k_{12},k_{23},k_{13},k_{123})$.

We set $X=\{\emptyset\}$. 
First, 
\begin{align}
\{s_1^{k_1} s_2^{k_2} s_3^{k_3}\mid k_1,k_2,k_3\in\mathbb N\}\subset \sSq(\partial H_3)_{\gamma}\cdot X \label{eq;inc123}
\end{align}
follows from 
\begin{align*}
\beta_2^{k_2}\beta_3^{k_3}\beta_{1}^{k_1}\cdot \emptyset=\beta_2^{k_2}\beta_3^{k_3}\cdot s_1^{k_1} = s_1^{k_1}s_2^{k_2} s_3^{k_3}\in \sSq(\partial H_3)_{\gamma}\cdot X.
\end{align*}
Next, we will show 
\begin{align}\label{eq-s13}
    \{s_1^{k_1} s_2^{k_2} s_3^{k_3} s_{13}^{k_{13}}\mid k_1,k_2,k_3,k_{13}\in\mathbb N\}\subset \sSq(\partial H_3)_{\gamma}\cdot X.
\end{align}
When $k_{13}=0$, 
$s_1^{k_1} s_2^{k_2} s_3^{k_3} s_{13}^{k_{13}}\in \sSq(\partial H_3)_{\gamma}\cdot X$ 
follows from \eqref{eq;inc123}. 
Suppose $k_{13}>0$ and 
$s_1^{k_1} s_2^{k_2} s_3^{k_3} s_{13}^{l}\in \sSq(\partial H_3)_{\gamma}\cdot X$ holds for $0\leq l<k_{13}$. 
We have 
$$\beta_3^{k_3}\beta_{13}^{k_{13}}\cdot (s_1^{k_1}s_2^{k_2}) = s_1^{k_1} s_2^{k_2} s_3^{k_3} s_{13}^{l} + \sum_{0\leq l<k_{13}}  s_1^{k_1} s_2^{k_2} s_3^{k_3} s_{13}^{l}\in \sSq(\partial H_3)_{\gamma}\cdot X.$$
From the assumption, we have $\sum_{0\leq l<k_{13}}  s_1^{k_1} s_2^{k_2} s_3^{k_3} s_{13}^{l}\in \sSq(\partial H_3)_{\gamma}\cdot X$. Thus we have 
$s_1^{k_1} s_2^{k_2} s_3^{k_3} s_{13}^{k_{13}}\in \sSq(\partial H_3)_{\gamma}\cdot X$.

Next, we will show 
\begin{align}\label{eq-s123}
\{s^{\vec{k}}\mid \vec{k}\in \Lambda,\ k_{12}=0\}\subset \sSq(\partial H_3)_{\gamma}\cdot X.
\end{align}
When $k_{123}=0$,
we have 
$$\beta_{23}^{k_{23}}\cdot (s_1^{k_1} s_2^{k_2} s_3^{k_3} s_{13}^{k_{13}}) = s_1^{k_1} s_2^{k_2} s_3^{k_3} s_{13}^{k_{13}} s_{23}^{k_{23}} \in \sSq(\partial H_3)_{\gamma}\cdot X.$$
Suppose $k_{123}>0$ and  $s_1^{k_1} s_2^{k_2} s_3^{k_3} s_{13}^{k_{13}} s_{23}^{k_{23}} s_{123}^{l}\in \sSq(\partial H_3)_{\gamma}\cdot X$ holds for any $0\leq l<k_{123}$.
We have 
$$\beta_2^{k_2}\beta_{23}^{k_{23}}(\beta_{123}')^{k_{123}}\cdot (s_1^{k_1}s_3^{k_3}s_{13}^{k_{13}}) = s_1^{k_1} s_2^{k_2} s_3^{k_3} s_{13}^{k_{13}} s_{23}^{k_{23}} s_{123}^{k_{123}} + \sum_{0\leq l<k_{123}}  s_1^{k_1} s_2^{k_2} s_3^{k_3} s_{13}^{k_{13}} s_{23}^{k_{23}} s_{123}^{l}\in \sSq(\partial H_3)_{\gamma}\cdot X.$$
From the assumption, we have $\sum_{0\leq l<k_{123}}  s_1^{k_1} s_2^{k_2} s_3^{k_3} s_{13}^{k_{13}} s_{23}^{k_{23}} s_{123}^{l}\in \sSq(\partial H_3)_{\gamma}\cdot X$. Thus we have 
$s_1^{k_1} s_2^{k_2} s_3^{k_3} s_{13}^{k_{13}} s_{23}^{k_{23}} s_{123}^{k_{123}}\in \sSq(\partial H_3)_{\gamma}\cdot X$.

By replacing the role of $\beta_{13}, s_{13}$, $\beta_{123}'$ with $\beta_{12},s_{12}$, $\beta_{123}$, similar to the above,  one can show
\begin{align}\label{eq-s123-s12}
\{s^{\vec{k}}\mid \vec{k}\in \Lambda,\ k_{13}=0\}\subset \sSq(\partial H_3)_{\gamma}\cdot X.
\end{align}


We will use the following Lemma for the rest of proof and will prove it in Section \ref{sec;pf_lem}.

\begin{lem}\label{keylem-s}
Let $\beta$ be a diagram in $\Sigma_{0}^4$ such that it intersects each $C_i$ (see Figure \ref{fig;edges_pants_g3}), $1\leq i\leq 6$, at $m_i$ points.
Suppose there exists a finite subset $\Lambda_0$ of $\Lambda$ 
such that 
\begin{equation}
\beta = 
\sum_{\vec{u}\in\Lambda_0} C_{\vec{u}}s^{\vec{u}},
\end{equation}
where $0\neq C_{\vec{u}}\in \bZ[q^{\pm1}]$ for any $\vec{u}\in\Lambda_0$.
Then, we have  ${\rm s}(\vec{u})\leq m_1+\cdots +m_6$ for any $\vec{u}\in\Lambda_0$. 
\end{lem}

Consider $\Lambda':=\{\vec k\in \Lambda\mid k_{23}=0 \}$.
In this paragraph, we will show 
\begin{align}\label{eq-s13-s12}
    \{s^{\vec{k}}\mid \vec{k}\in \Lambda'\}\subset \sSq(\partial H_3)_{\gamma}\cdot X.
\end{align}
using induction on ${\rm s}(\vec{k})$, defined in \eqref{s(vec)}. Trivially $s^{\vec k}=\emptyset\in \sSq(\partial H_3)_{\gamma}\cdot X$ when ${\rm s}(\vec k)= 0$.
Suppose 
\begin{align}\label{eq;assump}
\text{$s^{\vec k}\in \sSq(\partial H_3)_{\gamma}\cdot X$ when ${\rm s}(\vec k)<{\rm s}(\vec u) = l$ and $\vec k\in\Lambda'$,}
\end{align}
where $\vec u = (u_1,u_2,u_3,u_{12},u_{13},0,u_{123}) \in \Lambda'$. 
We have 
\begin{align*}
    \beta_{12}^{u_{12}}\cdot(s_1^{u_1} s_2^{u_2} s_3^{u_3} s_{13}^{u_{13}} s_{123}^{u_{123}}) \in \sSq(\partial H_3)_{\gamma}\cdot X.
\end{align*}
We can regard $\beta_{12}^{u_{12}}\cdot(s_1^{u_1} s_2^{u_2} s_3^{u_3} s_{13}^{u_{13}} s_{123}^{u_{123}})$  as a diagram in $\Sigma_{4}^0$ with crossings such that it intersects with $C_1\cup\cdots \cup C_6$ at ${\rm s}(\vec u)$ points. 
Then Lemma \ref{keylem-s} implies 
\begin{align}\label{eq-genus3-1}
    \beta_{12}^{u_{12}}\cdot(s_1^{u_1} s_2^{u_2} s_3^{u_3} s_{13}^{u_{13}} s_{123}^{u_{123}}) = \sum_{\vec v\in \Lambda', {\rm s}(\vec v) = {\rm s}(\vec u)} C_{\vec u,\vec v} s^{\vec v} + \sum_{\vec a\in \Lambda', {\rm s}(\vec a) < {\rm s}(\vec u)} C_{\vec u,\vec a} s^{\vec a} + \sum_{\vec b\in \Lambda\setminus \Lambda'} C_{\vec u,\vec b} s^{\vec b}.
\end{align}
The previous argument implies $\sum_{\vec b\in \Lambda\setminus \Lambda'} C_{\vec u,\vec b} s^{\vec b}\in \sSq(\partial H_3)_{\gamma}\cdot X$. The  assumption \eqref{eq;assump} implies $\sum_{\vec a\in \Lambda', {\rm s}(\vec a) < {\rm s}(\vec u)} C_{\vec u,\vec a} s^{\vec a}\in  \sSq(\partial H_3)_{\gamma}\cdot X$. 
Thus we have $\sum_{\vec v\in \Lambda', {\rm s}(\vec v) = {\rm s}(\vec u)} C_{\vec u,\vec v} s^{\vec v}\in \sSq(\partial H_3)_{\gamma}\cdot X$. 
By substituting $q=1$, equation \eqref{eq-genus3-1} will turn into
$\beta_{12}^{u_{12}}\cdot(s_1^{u_1} s_2^{u_2} s_3^{u_3} s_{23}^{u_{23}} s_{123}^{u_{123}}) = s^{\vec u}.$
Then the square matrix $(C_{\vec u,\vec v})_{\vec u\in\Lambda',{\rm s}(\vec u)=l;\; \vec v\in\Lambda',{\rm s}(\vec v)=l}$ is the identity matrix by substituting $q=1$. This implies that $(C_{\vec u,\vec v})_{\vec u\in\Lambda',{\rm s}(\vec u)=l;\; \vec v\in\Lambda',{\rm s}(\vec v)=l}$ is an invertible matrix in $\mathbb Q(q)$. Thus we have $s^{\vec u}\in \sSq(\partial H_3)_{\gamma}\cdot X$ for 
$\vec u\in\Lambda',{\rm s}(\vec u)=l$. 

From the previous argument, we have $s^{\vec{v}}\in \sSq(\partial H_3)_{\gamma}\cdot X$ for any $\vec{v}\in\Lambda$, i.e. $\iota_{*}(\sSq(\partial H_3)_{\gamma}\cdot X) = \sSq(H_3^{\gamma})$.

\begin{figure}[h]
      \centering
      \includegraphics[scale=0.45]{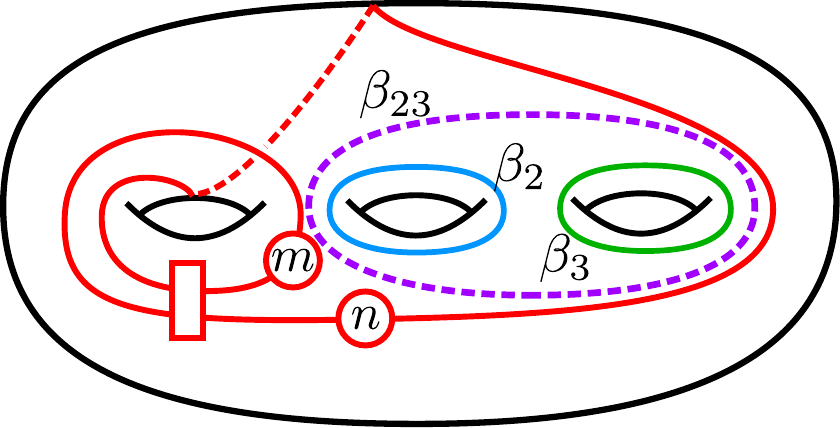}
      \caption{The red curve is the gluing curve $\gamma$. The  blue curve is $\beta_2$ the green curve is $\beta_{3}$, and the purple curve is $\beta_{23}$. Here, $\beta_{23}$ is on the back side.}\label{beta23}
  \end{figure}

\begin{figure}[h]
      \centering
      \includegraphics[scale=0.45]{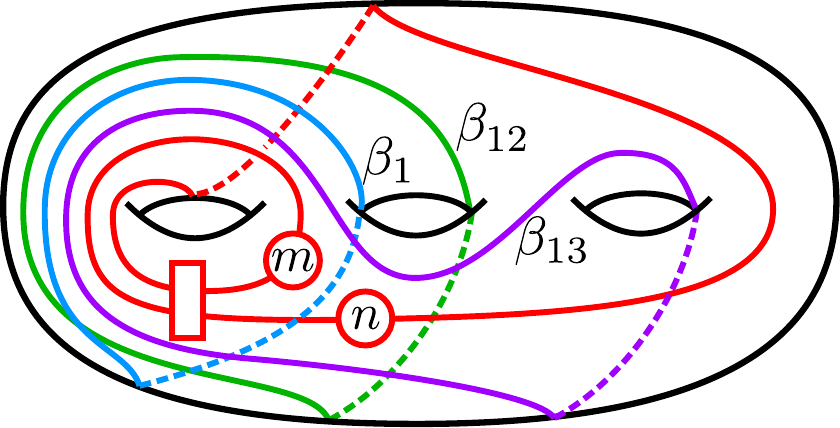}
      \caption{The red curve is the gluing curve $\gamma$. The  blue curve is $\beta_1$, the green curve is $\beta_{12}$, and the purple curve is $\beta_{13}$.}\label{beta1_12_13}
  \end{figure}

\begin{figure}[h]
      \centering
      \includegraphics[scale=0.45]{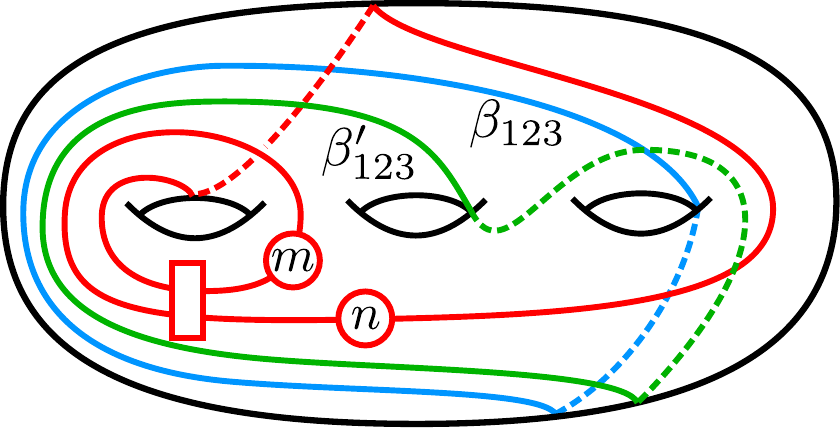}
      \caption{The red curve is the gluing curve $\gamma$. The blue curve is $\beta_{123}$ and the green curve is $\beta_{123}'$.}\label{beta123}
  \end{figure}
\end{proof}

\begin{rmk}
When $m$ and $n$ are coprime, using Seifert–van Kampen theorem, we have 
$$\pi_1(H_3^{\gamma})\cong \langle a,b,c\mid a^{m+n}b^{n}\rangle$$ for $\DT(\gamma)=(n+m,m,0,n,n,m+2n,t,0,0,0,0,0)$ in the second case in Case 2 of Theorem \ref{thm;example}. This implies that $H_3^{\gamma}\not\cong H_g$ for any $g$ and $H_3^{\gamma}$ is a non-trivial example with a boundary component of genus $2$.
Hence the higher genus case is not covered by \cite{AF22}, \cite{Le06}, \cite{Mar10}, \cite{LT14}, and \cite{LT15}. 
\end{rmk}

\section{Proofs of lemmas}\label{sec;pf_lem}

For any $\vec{k}\in\Lambda$, recall that $\text{deg}(s^{\vec{k}})$ is defined to be $\text{s}'(\vec{k})$, see (\ref{s(vec)}).

\begin{proof}[Proof of Lemma \ref{keylem}]
    Fix $\vec{k}=(k_1,k_2,k_3,k_{12},k_{13},k_{23},k_{123})\in \mathbb{N}^7$.
    From Lemma \ref{lem;basis}, there is $\Lambda_0\subset \Lambda$ satisfying the assumption.   
    Put ${k} = 
    \text{max}\{\deg(s^{\vec{u}})\mid \vec{u}\in\Lambda_0\}$. 
    We will show $k\leq {\rm s}'(\vec k)$.  Assume ${k}> {\rm s}'(\vec{k})$ on the contrary. 
    There is $\vec{u_0}\in\Lambda_0$ such that 
    $\deg(s^{\vec{u_0}}) = k$ and ${\rm s}(\vec{u_0}) = \max\{{\rm s}(\vec{u})\mid \deg(s^{\vec{u}}) = k,\vec{u}\in\Lambda_0\}$, where ${\rm s}(\vec{u})$ was defined as (\ref{s(vec)}).

    For any $\vec{u} = (u_1,u_2,u_3,u_{12},u_{13},u_{23},u_{123})\in\Lambda_0$, we have 
    $$s^{\vec{u}} = \sum_{\gamma\in{\rm Multi}{(\vec{u}})} C_{\gamma,\vec{u}} \gamma,$$ where ${\rm Multi}{(\vec{u}})$ is a finite subset of the set of isotopy classes of essential multicurves in $\Sigma_0^{4}$ and $C_{\vec{u},\gamma}\in \bZ[q^{\pm}]\setminus \{0\}$. 
    The following statement comes from e.g. \cite[Theorem 3.2]{AF17}.
    Using the symbols (\ref{vec(m)}) and (\ref{sum}), there is $\gamma_{\vec{u}}\in {\rm Multi}(\vec{u})$ such that $\vec{m}(\gamma_{\vec{u}})$ is equal to 
    \begin{align}
    u_1 \vec{m}(s_1)+ u_2 \vec{m}(s_2)+ u_3 \vec{m}(s_3)+u_{12} \vec{m}(s_{12})+u_{13} \vec{m}(s_{13})+u_{23} \vec{m}(s_{23})+u_{123} \vec{m}(s_{123})\label{sum_coord}
    \end{align}
    and, if $\gamma_{\vec{u}}\neq \gamma\in {\rm Multi}(\vec{u})$ then 
    ${\rm sum}(\gamma)< {\rm sum}(\gamma_{\vec{u}})$. 
    Note that ${\rm s}(\vec{u}) = {\rm sum}(\gamma_{\vec{u}})$.

    Suppose $\vec{u}\in\Lambda_0$ such that $\deg(s^{\vec{u}})<k$. 
    For any $\gamma\in{\rm Multi}(\vec{u})$, 
    $m_i(\gamma)\leq m_i(\gamma_{\vec{u}})\ (i=1,\dots,6)$. 
    We also have  
    $$m_1(\gamma_{\vec{u}}) + m_2(\gamma_{\vec{u}}) +m_4(\gamma_{\vec{u}}) +m_5(\gamma_{\vec{u}})\leq 
    \deg(s^{\vec{u}})<k=m_1(\gamma_{\vec{u_0}}) + m_2(\gamma_{\vec{u_0}}) +m_4(\gamma_{\vec{u_0}}) +m_5(\gamma_{\vec{u_0}}).$$ 
    Hence $\gamma\neq \gamma_{\vec{u_0}}$ for any $\gamma\in{\rm Multi}(\vec{u})$.  
    
    Suppose $\vec{u}\in\Lambda_0$ such that $\deg(s^{\vec{u}})=k$ and ${\rm s}(\vec{u})< {\rm s}(\vec{u_0})$. Trivially, we have $\gamma\neq \gamma_{\vec{u_0}}$ for any $\gamma\in{\rm Multi}(\vec{u})$. 
    
    Suppose $\vec{u}\in\Lambda_0$ such that $\deg(s^{\vec{u}})=k$, ${\rm s}(\vec{u}) = {\rm s}(\vec{u_0})$, and $\vec{u}\neq \vec{u_0}$. Lemma \ref{inde} implies $\gamma_{\vec{u}}\neq \gamma_{\vec{u_0}}$.
    For any $\gamma_{\vec{u}}\neq\gamma\in{\rm Multi}(\vec{u})$, we have 
    ${\rm sum}(\gamma)<{\rm sum}(\gamma_{\vec{u}} )= {\rm s}(\vec{u}) = {\rm s}(\vec{u_0})=
    {\rm sum}(\gamma_{\vec{u_0}} )$, which shows $\gamma\neq \gamma_{\vec{u_0}}$. 

     Now we expand \eqref{eqkey} as $$\alpha_{123}^{k_{123}}\alpha_{12}^{k_{12}}\alpha_{13}^{k_{13}}\alpha_{23}^{k_{23}}\cdot(s_1^{k_1}s_2^{k_2}s_3^{k_3}) = \sum_{\vec{u}\in\Lambda_0} C_{\vec{u}}s^{\vec{u}}=\sum_{\gamma\in{\rm Multi}(\vec{k})}C_{\vec{k},\gamma} \gamma,$$ where ${\rm Multi}(\vec{k})$ is a finite subset of the set of isotopy classes of essential multicurves on $\Sigma_0^{4}$ and $C_{\vec{k},\gamma}\neq 0$ for any $\gamma\in{\rm Multi}(\vec{k})$. Then the above argument implies that the most right-hand side contains $\gamma_{\vec{u_0}}$ and its coefficient is nonzero.
    Since the geometric intersection number does not increase when we resolve crossings, we have 
    $$m_1(\gamma) + m_2(\gamma) +m_4(\gamma) +m_5(\gamma) \leq \text{s}'(\vec{k})$$ for any $\gamma\in{\rm Multi}(\vec{k}).$ This contradicts the assumption 
    $$\deg(s^{\vec{u_0}})=m_1(\gamma_{\vec{u_0}}) + m_2(\gamma_{\vec{u_0}}) +m_4(\gamma_{\vec{u_0}}) +m_5(\gamma_{\vec{u_0}}) >\text{s}'(\vec{k}).$$ 
\end{proof}

\begin{lem}\label{lem;prod_sum_ss}
    Suppose $q=1$. For any two $\vec{u},\vec{v}\in\Lambda$, $s^{\vec{u}}s^{\vec{v}}$ is a linear sum of basis elements in $\{s^{\vec{k}}\mid \vec{k}\in \Lambda\}$ with degree less than or equal to $\deg({s^{\vec{u}}})+\deg({s^{\vec{v}}})$.
\end{lem}
Before we prove Lemma \ref{lem;prod_sum_ss}, we show the following lemma. 

\begin{lem}\label{lem;s12_13_23}
Suppose $q=1$. For any $k_{12},k_{13},k_{23}\in \bN$, we have 
\begin{align}
s_{12}^{k_{12}} s_{13}^{k_{13}} s_{23}^{k_{23}} =\sum_{\vec{u}\in\Lambda_0} C_{\vec{u}}s^{\vec{u}},\label{eq;sss=sum}
\end{align}
where $\Lambda_0$ is a finite subset of $\Lambda$ such that $\deg(s^{\vec{u}}) \leq 2k_{12}+3(k_{13}+k_{23})$ for any $\vec{u}\in\Lambda_0$.
\end{lem}
\begin{proof}

If $k_{12}k_{13}k_{23} = 0$, since $s_{12}^{k_{12}} s_{13}^{k_{13}} s_{23}^{k_{23}}$ is a basis element, $s_{12}^{k_{12}} s_{13}^{k_{13}} s_{23}^{k_{23}}$ satisfies the claim. 

If $k_{12}k_{13}k_{23}\neq 0$ then we apply Lemma \ref{relation} for $s_{12}^{{k_{12}}-1} s_{13}^{{k_{13}}-1} s_{23}^{{k_{23}}-1}s_{12} s_{13} s_{23}$. 
Since each term on the right-hand side of the equation in Lemma \ref{relation} has the degree no greater than $\deg(s_{12}s_{13}s_{23})$. 
By applying Lemma \ref{relation} repeatedly until we have (\ref{eq;sss=sum}), we can conclude that $s_{12}^{k_{12}} s_{13}^{k_{13}} s_{23}^{k_{23}}$ is a linear sum of basis elements whose degrees are not greater than $2{k_{12}}+3({k_{13}}+{k_{23}})$.
\end{proof}

\begin{proof}[Proof of Lemma \ref{lem;prod_sum_ss}]
If $s^{\vec{u}}s^{\vec{v}}$ does not contain $s_{12} s_{13} s_{23}$, the claim is trivial. If $s^{\vec{u}}s^{\vec{v}}$ contains $s_{12} s_{13} s_{23}$, the claim follows from  Lemma \ref{lem;s12_13_23}.
\end{proof}

\begin{proof}[Proof of Lemma \ref{lem;prod_sum_sss}]
To prove the claim, we use mathematical induction on $k_{12}$. 
By applying Lemma \ref{relation} and Lemma \ref{lem;prod_sum_ss} to 
$s_{12} s_{13}^{k_{13}} s_{23}^{k_{23}}$, it is equal to
\begin{align*}
&(s_{13}s_2s_{123}+ s_{23}s_1s_{123}+\text{terms with degrees less than $8$})s_{13}^{{k_{13}}-1} s_{23}^{{k_{23}}-1}\\
&= s_{13}^{{k_{13}}} s_{23}^{{k_{23}}-1} s_2s_{123}+ s_{13}^{{k_{13}}-1} s_{23}^{{k_{23}}}s_1s_{123}+\text{(terms with degrees less than $2+3({k_{13}}+{k_{23}}$})).
\end{align*}
Thus the claim holds for $s_{12} s_{13}^{k_{13}} s_{23}^{k_{23}}$.

Fix $k\geq 2$ and suppose the claim holds for any ${k_{12}}< k$. Then 
\begin{align*}
s_{12}^{k} s_{13}^{k_{13}} s_{23}^{k_{23}} 
&=s_{12} \big(\sum_{\vec{u}\in\Lambda'} C_{\vec{u}}s^{\vec{u}} + (\text{terms with degrees less than $2(k-1) + 3({k_{13}}+{k_{23}})$})\big)\\
&= \sum_{\vec{u}\in\Lambda'} C_{\vec{u}} s_{12}s^{\vec{u}} +(\text{terms with degrees less than $2k + 3({k_{13}}+{k_{23}}))$},
\end{align*}
where $\Lambda'$ is a finite subset of $\Lambda$ such that, for any $\vec{u}=(u_1,u_2,u_3,u_{12},u_{13},u_{23},u_{123})\in\Lambda'$, we have $\deg(s^{\vec{u}}) = 2(k-1) + 3({k_{13}}+{k_{23}})$ and $u_{13}u_{23}<{k_{13}}{k_{23}}$. 
For $\vec{u}\in\Lambda'$, if $u_{13}u_{23} = 0$ then $s_{12}s^{\vec{u}}$ is a basis element and $u_{13}u_{23}=0< {k_{13}}{k_{23}}$. If $u_{13}u_{23}\neq 0$, we have 
\begin{align*}
s_{12}s^{\vec{u}} 
&=s_1^{u_1}s_2^{u_2}s_3^{u_3}s_{123}^{u_{123}}s_{12}s_{13}^{u_{13}}s_{23}^{u_{23}}\\
&=s_1^{u_1}s_2^{u_2}s_3^{u_3}s_{123}^{u_{123}} \big(\sum_{\vec{u}\in\Lambda''}C_{\vec{v}}s^{\vec{v}} + (\text{terms with degrees less than $2 + 3(u_{13}+u_{23})$})\big),
\end{align*}
where $\Lambda''$ is a finite subset of $\Lambda$ such that, for any $\vec{v}=(v_1,v_2,v_3,v_{12},v_{13},v_{23},v_{123})\in\Lambda''$, we have $\deg(s^{\vec{v}}) = 2 + 3(u_{13}+u_{23})$ and $v_{13}v_{23}<u_{13}u_{23}<k_{13}k_{23}$. Then 
\begin{align*}
s_{12}s^{\vec{u}} = \sum_{\vec{u}\in\Lambda''}C_{\vec{v}}s_1^{u_1}s_2^{u_2}s_3^{u_3}s_{123}^{u_{123}}s^{\vec{v}} + (\text{terms with degrees less than $2 + \deg(s^{\vec{u}})$}).
\end{align*}
Hence the claim holds for ${k_{12}} = k$.
\end{proof}

\begin{proof}[Proof of Lemma \ref{keylem-s}]
Assume that ${\rm s}(\vec{u})> m_1+\cdots +m_6$ for some $\vec{u}\in\Lambda_0$ on the contrary. Then 
$m=\text{max}\{{\rm s}(\vec u)\mid \vec u\in\Lambda_0\}>m_1+\cdots+m_6$.
As in the proof of Lemma \ref{keylem}, for any $\vec{u} = (u_1,u_2,u_3,u_{12},u_{13},u_{23},u_{123})\in\Lambda_0$, we have 
    $$s^{\vec{u}} = \sum_{\gamma\in{\rm Multi}{(\vec{u}})} C_{\gamma,\vec{u}} \gamma.$$
There is $\gamma_{\vec{u}}\in {\rm Multi}(\vec{u})$ such that $\vec{m}(\gamma_{\vec{u}})$ is equal to the formula in \eqref{sum_coord}.
Take $\vec u_0\in\Lambda_0$ such that ${\rm s}(\vec u_0) = m$. 

For any $\vec v\in\Lambda_0$ such that ${\rm s}(\vec v)<m$, we have 
${\rm s}(\gamma)<{\rm s}(\gamma_{\vec u_0})$ for $\gamma\in {\rm Multi}(\vec v)$. Thus $\gamma\neq \gamma_{\vec u_0}.$
For any $\vec v\in\Lambda_0$ such that ${\rm s}(\vec v)=m$, we have 
${\rm s}(\gamma)<{\rm s}(\gamma_{\vec u_0})$ for $\gamma\in {\rm Multi}(\vec v)\setminus\{\gamma_{\vec v}\}$. 
Thus $\gamma\neq \gamma_{\vec u_0}$ for $\gamma\in {\rm Multi}(\vec v)\setminus\{\gamma_{\vec v}\}$.
Lemma \ref{inde} implies $\gamma_{\vec{v}}\neq \gamma_{\vec{u_0}}$ for $\vec v\in\Lambda_0\setminus\{\vec{u}_0\}$.
Thus if we present $\sum_{\vec{u}\in\Lambda_0} C_{\vec{u}}s^{\vec{u}}$ as a linear sum of essential multicurves, it contains $\gamma_{\vec u_0}$.

For a diagram $\alpha$, by resolving its crossings, $\alpha$ can be presented as a linear sum of essential multicurves. This implies that this linear sum does not contain $\gamma_{\vec u_0}$ since the intersection number between $\gamma_{\vec u_0}$ and $C_1\cup\cdots\cup C_6$ is strictly greater than $m_1+\cdots+m_6$. 
Then we have a contradiction. 
\end{proof}

\appendix\section{Sphere sliding}\label{sec:sliding}

We recall the sphere sliding given in \cite{HP95}. 
The \textbf{degree} of a skein in $M_1\#_{\bf D}M_2$ is the geometric intersection number with the skein and $D_1$ (equivalently that of the skein and $D_2$). 
\begin{align}
&\begin{array}{c}\includegraphics[scale=0.32]{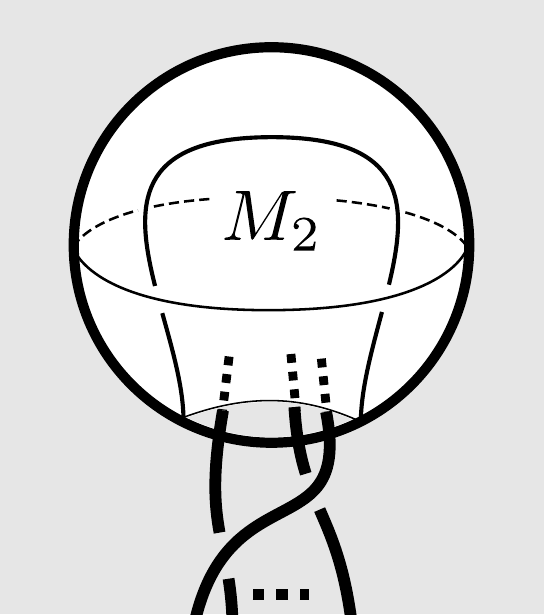}\end{array}=q^n\begin{array}{c}\includegraphics[scale=0.32]{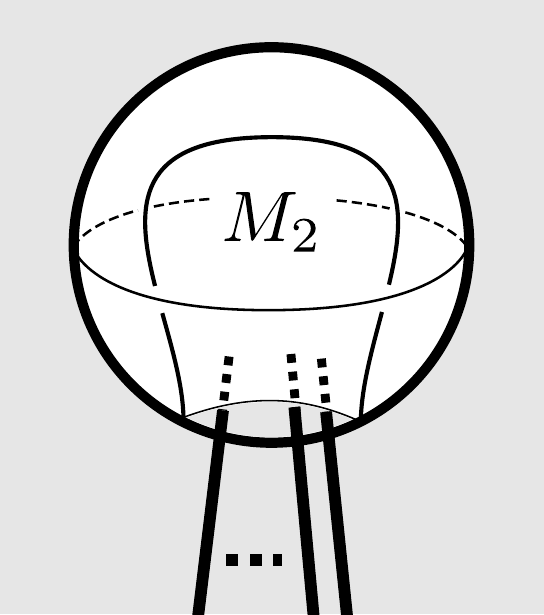}\end{array}
+\text{lower degree terms},\label{eq;sphere_posi} \\
&\begin{array}{c}\includegraphics[scale=0.32]{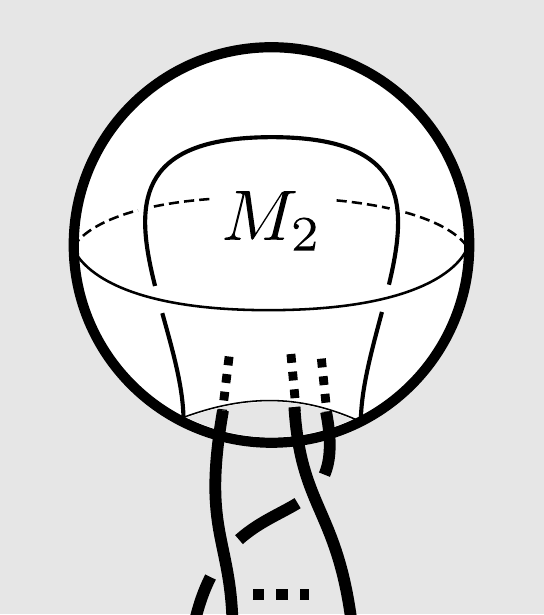}\end{array}=q^{-n}\begin{array}{c}\includegraphics[scale=0.32]{draws/sphere_standard.pdf}\end{array}
+\text{lower degree terms}, \label{eq;sphere_nega}
\end{align}
where the shaded part outside of the sphere is the inside of $M_1^{\cC_1}$ and the shaded part on the sphere is the attaching region $D_1$.
\begin{align}
\begin{array}{c}\includegraphics[scale=0.32]{draws/sphere_posi.pdf}\end{array}
=\begin{array}{c}\includegraphics[scale=0.32]{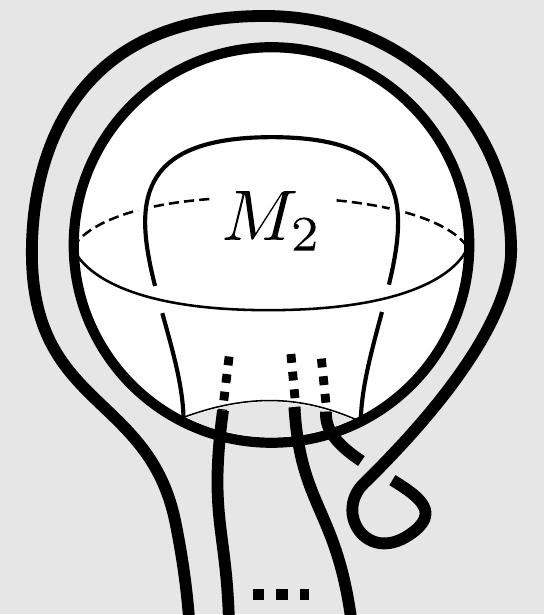}\end{array}
=\begin{array}{c}\includegraphics[scale=0.32]{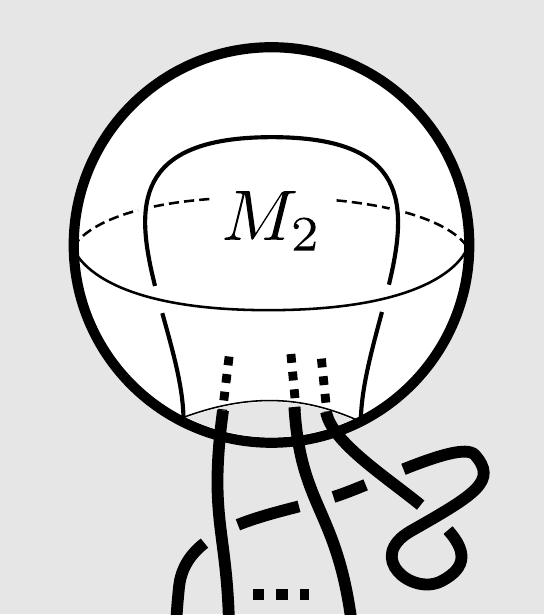}\end{array}=q^{-6}\begin{array}{c}\includegraphics[scale=0.32]{draws/sphere_nega.pdf}\end{array}, \label{eq;sphere_sliding}
\end{align}
where the last equality follows from $\begin{array}{c}\includegraphics[scale=0.15]{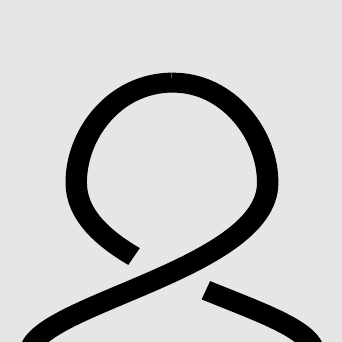}\end{array}
=-q^{-3}\begin{array}{c}\includegraphics[scale=0.15]{draws/empty_loop.pdf}\end{array}.$ Note that a similar idea is originally given in \cite{HP95}, where we call the technique the \textbf{sphere sliding}. 

From (\ref{eq;sphere_sliding}) with (\ref{eq;sphere_posi}) and (\ref{eq;sphere_nega}), 
\begin{align}
\begin{array}{c}\includegraphics[scale=0.32]{draws/sphere_standard.pdf}\end{array}
=\dfrac{1}{q^n-q^{-n-6}}(\text{lower degree terms}). \label{eq;sphere_decrease}
\end{align}
Note that (\ref{eq;sphere_decrease}) implies the degree of a framed link in $M_1\#_{\bf D}M_2$ can decrease, where the degree of a framed link in $M_1\#_{\bf D}M_2$ is the geometric intersection number of the framed link and the properly embedded disk $D$. 
Since the degree is even, one can apply similar way repeatedly until all the degrees are 0.

\end{document}